\documentclass[a4paper, leqno]{amsart}
\title{A Lie algebraic pattern behind logarithmic CFTs}

\author{Shoma Sugimoto}
\address{Division of Physics,
Mathematics and Astronomy, California Institute of Technology}
\email{shomasugimoto361@gmail.com}

\author{Hao Li}
\address{Chiral Representation Theory Unit, Okinawa Institute of Science and Technology (OIST)}
\email{hao.lilwh@gmail.com}

\usepackage[radius=10mm,
,mark=o]{dynkin-diagrams}

\usepackage{latexsym}
\usepackage{amsmath}
\usepackage{amsfonts}
\usepackage{amssymb}
\usepackage{amsxtra}
\usepackage{mathrsfs}
\usepackage{eucal}
\usepackage[all]{xy}
\usepackage{color}
\usepackage{braket}
\usepackage{graphics, setspace}
\usepackage{etoolbox, textcomp}
\usepackage{comment}
\usepackage{changepage}
\usepackage{xcolor}
\definecolor{rouge}{rgb}{0.85,0.1,.4}
\definecolor{bleu}{rgb}{0.1,0.2,0.9}
\definecolor{violet}{rgb}{0.7,0,0.8}
\usepackage[colorlinks=true,linkcolor=bleu,urlcolor=violet,citecolor=rouge]{hyperref}
\usepackage{cleveref}
\usepackage{lscape}
\usepackage{tikz}
\usepackage{eqname}

\theoremstyle{definition}
\newtheorem{definition}{Definition}[section]
\newtheorem{remark}[definition]{Remark}
\newtheorem{example}[definition]{Example}

\theoremstyle{plain}

\newtheorem{theorem}
[definition]
{Theorem}
\newtheorem{corollary}[definition]{Corollary}
\newtheorem{lemma}[definition]{Lemma}
\newtheorem{conjecture}[definition]{Conjecture}

\numberwithin{equation}{section}

\newcommand{\Z}{\mathbb{Z}}
\newcommand{\R}{\mathbb{R}}
\newcommand{\C}{\mathbb{C}}

\newcommand{\ch}{\operatorname{ch}}
\newcommand{\sch}{\operatorname{sch}}

\newcommand{\ad}{\mathrm{ad}}

\newcommand{\g}{\mathfrak{g}}

\newcommand{\h}{\mathfrak{h}}

\newcommand{\sll}{\mathfrak{sl}}
\newcommand{\osp}{\mathfrak{osp}}

\newcommand{\cryov}[2]{{#1}\uparrow{#2}}

\def\beq#1\eeq{\begin{align}#1\end{align}}
\newcommand{\bi}{\begin{itemize}}
\newcommand{\ei}{\end{itemize}}

\usepackage{mathabx}

\begin{document}

\maketitle

\begin{abstract}
We introduce a purely Lie algebraic formalization of the Feigin--Tipunin's geometric construction of logarithmic CFTs/VOAs.
After reformulating the geometric representation theory of FT construction under this new setting, within this framework, we uniformly construct the (multiplet) principal W-algebras at positive integer level associated with any simple Lie algebra $\g$ and Lie superalgebra $\osp(1|2r)$, thereby establishing Weyl-type character formulas and simplicity theorems that extend the first author’s previous results.
\end{abstract}

\section{Introduction}
\label{section: introduction}
The study of logarithmic conformal field theory (LCFT) has become increasingly important in recent years.
However, LCFT is much more complicated than rational cases and has not been well studied.
In 2010, Feigin--Tipunin \cite{FT} proposed a geometric construction of LCFTs, known as the
\textit{Feigin--Tipunin (FT) construction}, defined as the $0$-th sheaf cohomology $H^0(G\times_BV)$ of $G$-equivariant VOA bundle $G\times_BV$ over the flag variety $G/B$. 
They further claimed that
\textit{multiplet W-algebras}\footnote{In the case of $\g=\sll_2$ with the principal nilpotent element, this is one of the most famous and well studied LCFTs called \textit{triplet Virasoro algebra} (see e.g. \cite{AM}). Here the ``multiplet W-(super)algebra" is intended to be a generic term for variations of the triplet Virasoro algebra (e.g., the cases of singlet/doublet, higher rank, general W-algebra, etc.).} can be constructed and studied using their method.
The first author \cite{Sug,Sug2} provided
rigorous proofs of several claims made in \cite{FT}.
In \cite{CNS}, the authors investigated the $V^{(p)}$-algebra ($=$ doublet affine $\sll_2$) \cite{A1,ACGY} by combining the method in \cite{Sug,Sug2} with the inverse quantum Hamiltonian reduction \cite{A2}.
On the other hand, a close examination of \cite{Sug,Sug2} suggests that, although formulated in the framework of VOAs, the essential arguments ultimately reduce to properties of an underlying Lie algebraic setting.
As detailed below, extracting such a Lie algebraic setting and argument not only provides a clear perspective for the study of LCFTs, 
but also sheds light 
\cite{Sug3,Sug4} 
on the expected correspondence between LCFTs and $3$-manifolds via the \textit{$\hat{Z}$-invariants} \cite{CCFGH,CCKPS,CFGHP,GPPV}. 

In Part \ref{part 1}, we introduce such a purely Lie algebraic setting, tentatively named \textit{shift system}, and prove that the main results of \cite{Sug} can be understood and derived from the setting (Theorem \ref{main theorem of part 1}). 
In other words, if a VOA-module $V$ (e.g. irreducible lattice VOA-module) fits into a shift system, the corresponding module over the multiplet W-algebra satisfies Theorem \ref{main theorem of part 1}.
In particular, when it admits the FT construction $H^0(G\times_BV)$, it can
be studied using methods from geometric representation theory.
As demonstrated in Section \ref{section:construction of shift system}, the verification of the axiom of shift system is basically reduced to a relatively easy computation ($=$ Serre relation of long screening operators) on $V$ and the rank $1$ case ($=$ ``Felder complex" of short screening operators). 
Hence, the shift system provides a useful framework for reducing the study of higher rank LCFTs to much easier cases: the rank $1$ cases and free field algebras.
Moreover, with some information about the corresponding ``W-algebra" $H^0(G\times_BV)^G$ (e.g. Kazhdan--Lusztig decomposition), the main results of \cite{Sug2} are proved (Section \ref{section: simplicity revisit}).
Using these results, in Part \ref{part 2}, we will give the shift system and the FT constructions corresponding to the (multiplet) principal W-(super)algebras $\mathbf{W}^k(\g)$ for any simple Lie algebra $\g$ and Lie superalgebra $\g=\mathfrak{osp}(1|2n)$, respectively, and prove almost all the main results of \cite{Sug,Sug2} for these new cases (Theorem \ref{main theorem of part 2}).
At first glance, these results themselves may appear to be just a straightforward extension of \cite{Sug,Sug2}.
However, as mentioned briefly at the end of the last paragraph, the value of this paper lies in a \textit{change of perspective}, which greatly increases the transparency and flexibility of the discussion in LCFTs, by focusing not on individual examples of LCFT and their complicated VOA-structures but on a simple Lie algebraic pattern behind them in common.
In that sense, the application to the two examples in Part \ref{part 2} is a demonstration of this new perspective which will itself be further developed and applied to a much broader class of expected LCFTs corresponding to $3$-manifolds in the future (see the last of introduction below).

Let us describe the overview of this paper.
For a finite dimensional simple Lie algebra $\g$ of rank $r$, we consider a triple $(\Lambda,\uparrow,\{V_\lambda\}_{\lambda\in\Lambda})$, named \textit{shift system} (see Definition \ref{Def FT data} for the full axioms).
Here 
$\Lambda$ has the action $\ast\colon W\times\Lambda\rightarrow\Lambda$ of the Weyl group $W$ of $\g$,
$\uparrow\colon W\times\Lambda\rightarrow\mathfrak{h}^\ast$ is a map (\textit{shift map}) satisfying a few simple axioms, and $V_\lambda$ is a graded weight $B$-module with a parameter $\lambda\in\Lambda$ and  
$B$-module homomorphisms called \textit{short screening operators} $Q_{i,\lambda}\colon V_\lambda\rightarrow V_{\sigma_i\ast\lambda}(\sigma_i\uparrow\lambda)$ satisfying the
``Felder complex" \cite{Fel} for each direction $i\in I:=\{1,\dots,r\}$, where $(\sigma_i\uparrow\lambda)$ means the shift of Cartan weight. 
Recall that in \cite{FT,Sug,Sug2}, $V_\lambda$ is an irreducible module $V_{\sqrt{p}(Q+\lambda)}$ over the lattice VOA $V_{\sqrt{p}Q}$ with a conformal grading $V_\lambda=\bigoplus_{\Delta}V_{\lambda,\Delta}$ and a parameter 
$\lambda\in\Lambda
\simeq
(\sqrt{p}Q)^\ast/\sqrt{p}Q
\simeq
\tfrac{1}{p}Q^\ast/Q$, and the long/short screening operators define $B$- and $W$-actions on $V_\lambda$ and $\Lambda$, respectively.
The shift of Cartan weights by the short screening operators satisfies the axioms of the shift map above.
The situation in \cite{CNS} is a little more complicated, but essentially the same.
Therefore, the shift system can be regarded as an abstraction of the relationship between such VOA-modules and screening operators acting on them, forgetting the VOA-structure, and the shift map $\uparrow\colon W\times\Lambda\rightarrow P$ records how the
Cartan weight is shifted when short screening operators are applied along a Weyl group element.
The first main theorem of Part \ref{part 1} asserts that all the main results of \cite{Sug} hold even under such abstraction.

\begin{theorem}\label{main theorem of part 1}
Let $(\Lambda,\uparrow,\{V_\lambda\}_{\lambda\in\Lambda})$ be a shift system (see Definition \ref{Def FT data}).
Recall that the \textit{evaluation map}
$\mathrm{ev}\colon H^0(G\times_B V_\lambda)\to V_\lambda$,
$s\mapsto s(\mathrm{id}_{G/B})$,
sends a global section of the bundle $G\times_B V_\lambda$ over $G/B$ to its
value at the base point $\mathrm{id}_{G/B}=B/B$.
\begin{enumerate}
\item\label{main theorem 1(1)}
\textbf{(Feigin--Tipunin construction)}
The evaluation map defines an injection
\begin{align*}
\operatorname{ev}\colon H^0(G\times_BV_\lambda)\rightarrow\bigcap_{i\in I}\ker Q_{i,\lambda}|_{V_\lambda},\ \ s\mapsto s(\mathrm{id}_{G/B})
\end{align*}
and it is an isomorphism iff $\lambda\in\Lambda$ satisfies the following ``weak condition":
\begin{align}
\text{For any $(i,j)\in I\times I$, $\lambda\in\Lambda^{\sigma_j}$ or $(\cryov{\sigma_j}{\lambda},\alpha_i^\vee)=-\delta_{ij}$.}
\tag{weak}\label{weak condition}
\end{align}
Note that the image of $H^0(G\times_BV_\lambda)$ is the maximal $G$-submodule\footnote{Namely, the maximal $B$-submodule such that its $B$-action can be extended to the $G$-action.}
of $V_\lambda$.
\item\label{main theorem 1(2)}
\textbf{(Borel--Weil--Bott-type theorem)}
For a minimal expression \[w_0=\sigma_{i_{l(w_0)}}\cdots\sigma_{i_1}\sigma_{i_0}\] of the longest element $w_0$ of $W$ (where $\sigma_{i_0}=\operatorname{id}$ for convenience), if $\lambda\in\Lambda$ satisfies the following ``strong condition":
\begin{align}
\text{
For any $0\leq m\leq N-1$,
$(\sigma_{i_m}\cdots\sigma_{i_0}\uparrow\lambda,\alpha_{i_{m+1}}^\vee)=0$,
}
\tag{strong}\label{strong condition}
\end{align}
then we have a natural $G$-module isomorphism
\begin{align*}
H^n(G\times_BV_\lambda)\simeq H^{n+l(w_0)}(G\times_BV_{w_0\ast\lambda}(w_0\uparrow\lambda)).
\end{align*}
\end{enumerate}
\end{theorem}
Note that the weak/strong conditions in Theorem \ref{main theorem of part 1} are special cases of Definition \ref{Def FT data}\eqref{length and carry-over}.
That is, these conditions  are constraints on the shifts of Cartan weights.

Since the constructions and homomorphisms in Theorem \ref{main theorem of part 1} are natural, they are compatible with additional algebraic (e.g, VOA-module) structures.
In particular, if $V_0$ is a VOA and $V_\lambda$ are $V_0$-modules, then $H^0(G\times_BV_0)$ has the induced VOA-structure and $H^n(G\times_BV_\lambda)$ are $H^0(G\times_BV_0)$-modules (see \cite[Section 2.1, Corollary 2.21]{Sug2}).
It allows us to develop a geometric representation theory of multiplet W-(super)algebras, traditionally defined in the form of the right-hand side of Theorem \ref{main theorem of part 1}\eqref{main theorem 1(1)}.
Except for rank 1 cases, these VOAs are difficult to study algebraically due to their complicated VOA structures.
However, the geometric construction allows us to avoid the bottleneck and obtain a variety of results qualitatively, as in \cite{Sug,Sug2,CNS}.

In Part \ref{part 2}, we demonstrate the usefulness of the shift system and Theorem \ref{main theorem of part 1} by applying them to specific VOSAs.
Let us consider the rescaled root lattice $\sqrt{p}Q$ for some $p\in\Z_{\geq 1}$.
We consider the case where $\sqrt{p}Q$ is positive-definite integral, namely, $p\in r^\vee\Z_{\geq 1}$ for the lacing number $r^\vee$ or $p$ is odd and $\g=B_r$.
In the first case, $\sqrt{p}Q$ is even and we consider $V_\lambda=V_{\sqrt{p}(Q+\lambda)}$.
On the other hand, in the second case, $\sqrt{p}Q$ is odd and we need a modification $V_\lambda=V_{\sqrt{p}(Q+\lambda)}\otimes F$ by the free fermion $F$.
In Section \ref{section:construction of shift system}, from minimal natural assumptions \eqref{setup in the non-super case} or \eqref{setup in the super case}, shift systems $(\Lambda,\uparrow,\{V_\lambda\}_{\lambda\in\Lambda})$ are uniquely constructed (Theorem \ref{main theorem in Section 2}).
As noted in the second paragraph above, only relatively easy computations on $V_\lambda$ and some basic results in rank $1$ cases (\cite{AM} and \cite{AM1,AM2}) are used here.
Therefore, Theorem \ref{main theorem of part 1} is applied to our cases (for more detail, see Section \ref{section: main results}).

Let us explain the two consequences of Theorem \ref{main theorem of part 1}\eqref{main theorem 1(2)}, namely, the \textit{Weyl-type character formula} and the \textit{simplicity theorem} (see Section \ref{subsection: Weyl-type character formula} and \ref{section: simplicity revisit}).
The natural $G$-action on $H^0(G\times_BV_\lambda)$ gives the decomposition
\begin{align}
\label{the decomposition}
H^0(G\times_BV_\lambda)
\simeq 
\bigoplus_{\beta\in P_+}
L_\beta\otimes\mathcal{W}_{-\beta+\lambda},
\end{align}
where $L_\beta$ is the irreducible $\g$-module with highest weight $\beta$, and $\mathcal{W}_{-\beta+\lambda}$ is the multiplicity of a weight vector of $L_\beta$.
In Part \ref{part 2}, we consider additional VOA structures, where $\mathcal{W}_0\simeq H^0(G\times_BV_0)^G$ is a subalgebra of the multiplet W-(super)algebra $H^0(G\times_BV_0)$ and $\mathcal{W}_{-\beta+\lambda}$ is a $\mathcal{W}_0$-modules.
If $\lambda$ satisfies \eqref{strong condition}, then by combining $H^{n>0}(G\times_BV_\lambda)\simeq 0$ with the Atiyah--Bott localization formula \cite{AB}, we obtain the Weyl-type character formula
\begin{align}\label{character formula in introduction}
\ch_q H^0(G\times_BV_\lambda)
&=
\sum_{\beta\in P_+}\operatorname{dim}L_{\beta}
\ch_q\mathcal{W}_{-\beta+\lambda}\\
&=
\sum_{\beta\in P_+}\operatorname{dim}L_{\beta}
\sum_{\sigma\in W}(-1)^{l(\sigma)}\ch_q V_{\sigma\ast\lambda}^{h=\beta-\sigma\uparrow\lambda},
\end{align}
where $V_\mu^{h=\gamma}$ is the subspace of $V_\mu$ at the Cartan weight $h=\gamma$.
On the other hand, by combining the case $n=0$ with Serre duality, we have $H^0(G\times_BV_\lambda)\simeq H^0(G\times_BV_{\lambda'})^\ast$ for some $\lambda'\in\Lambda$.
In the VOA side, it leads to the self-duality and simplicity of the vacuum case $H^0(G\times_BV_0)$.
By the quantum Galois theory \cite{DM, McR2}, $\mathcal{W}_{-\beta}$ are also simple as $\mathcal{W}_0$-modules.
In addition, if a Kazhdan--Lusztig-type character formula for the simple quotient of $\mathcal{W}_{-\beta+\lambda}$ is known, by comparing it with \eqref{character formula in introduction}, we can extend the simplicity to the whole $\lambda\in\Lambda$ satisfying \eqref{strong condition}.
For more detail, see Section \ref{section: simplicity revisit}.

Let us go back to our special cases above.
In our cases, $\mathcal{W}_0$ is the principal W-(super)algebra $\mathbf{W}^k(\g)$ or $\mathbf{W}^k(\osp(1|2r))$, and thus we have the Kazhdan--Lusztig type character formula derived from \cite{KT,KT1}  
and the exactness of $+$-reduction $H^0_{\mathrm{DS},+}(\cdot)$ in \cite{Ar} (however, in the second case we assume the latter).
By applying the discussion in the last paragraph, 
we can extend the main results of \cite{ArF} to our cases
(Theorem \ref{simplicity theorem in non-super case} and \ref{simplicity theorem in super case}).

\begin{theorem}
\label{main theorem of part 2}
Let us consider the setup in Theorem \ref{main theorem in Section 2} 
and Section \ref{subsection: preliminary from W-algebra}.

Then for each case, we have 
$\mathcal{W}_0\simeq \mathbf{W}^{k}(\g)\simeq\mathbf{W}^{\check{k}}({}^L\g)$
and 
$\mathcal{W}_0\simeq \mathbf{W}^{k}(\osp(1|2r))\simeq\mathbf{W}^{\check{k}}(\osp(1|2r))$,
respectively.
Furthermore, in the first case, for any $\alpha\in P_+\cap Q$ and $\lambda\in\Lambda$ such that $(p\lambda_\bullet+\rho^\vee,{}^L\theta)\leq p$, $\mathcal{W}_{-\alpha+\lambda}$ and $H^0(G\times_BV_\lambda)$ are simple as $\mathcal{W}_0$- and $H^0(G\times_BV_0)$-modules, respectively.
In the second case, under similar conditions with the restriction $\lambda^\bullet=0$ and the assumption that the $+$-reduction is exact, the same simplicity theorem holds.
In the first case, the above simplicity of $\mathcal{W}_{-\alpha+\lambda}$ also leads to the simplicity and duality of the Arakawa--Frenkel modules \cite{ArF}
$\check{\mathbf{T}}^{m}_{p\lambda_\bullet,\alpha+\lambda^\bullet}
\simeq
\mathbf{T}^{1/p}_{\alpha+\lambda^\bullet,p\lambda_\bullet}$.
\end{theorem}

The decomposition \eqref{the decomposition} with Theorem \ref{main theorem of part 2} is regarded as the Schur-Weyl type duality for the multiplet principal W-algebras. 
In the non-simply laced case, in light of \cite{McR2}, the tensor category of $\mathbf{W}^k(\mathfrak{g})$-modules is expected to exhibit a symmetry governed by the Langlands dual Lie algebra ${}^L\mathfrak{g}$.

Finally, let us discuss two directions to which Theorem \ref{main theorem of part 1} points.
Our geometric representation theory of FT construction has similar aspects to the \textit{modular representation theory}.
Indeed, in our special cases, the strong condition in Lemma \ref{lemm: strong condition for our cases}\eqref{lemm: strong condition for non-super case} also appears in the BWB theorem for ${}^LG$ \cite[II.5]{Jan}.
Ultimately, such a similarity should come from the \textit{log Kazhdan--Lusztig correspondence} (e.g. \cite{FGST,FT,NT,MY,GN,CLR}): an expected categorical equivalence between a multiplet $W$-algebra and a \textit{small quantum group}.
For example, in our case, $H^0(G\times_BV_{\sqrt{p}Q})$ is expected to correspond to the small quantum group $u_{\zeta}({}^L\g)$ at $\zeta=\frac{\pi i}{p}$.\footnote{On the other hand, our super case is expected to correspond to some small quantum group at $\zeta=\tfrac{2\pi i}{2m-1}$ (see \cite{AM1}).}
Although the case $\mathfrak{g}=\mathfrak{sl}_2$
 was already treated in \cite{GN}, it seems necessary to develop both our geometric methods and W-algebraic approaches for higher rank generalizations (see Remark \ref{Section4-Remark4.4}).

On the other hand, under the correspondence of the BWB-type theorems above, Theorem \ref{main theorem of part 1}\eqref{main theorem 1(1)} clearly does not hold in the modular representation theory \cite{Jan}.
In this regard, the first author conceived the following idea a few years ago: 
Let $\vec{\lambda}=(\lambda_1,\dots,\lambda_N)\in\Lambda_1\times\cdots\times\Lambda_N$ be a parameter sequence and $\mathbb{V}_{\vec{\lambda}}$ be a $G$-module.
We consider the situation such that for each $0\leq n\leq N-1$, an appropriate quotient 
\begin{align*}
\underbrace{\tilde{H}^0(G\times_B\tilde{H}^0(G\times_B\cdots\tilde{H}^0(G\times_B}_{\text{$n$-times}}\tilde{\mathbb{V}}_{\vec{\lambda}})\cdots))  
\end{align*}
of
\begin{align*} 
\underbrace{{H}^0(G\times_B\tilde{H}^0(G\times_B\cdots\tilde{H}^0(G\times_B}_{\text{$n$-times}}\tilde{\mathbb{V}}_{\vec{\lambda}})\cdots))
\end{align*} 
is a shift system with respect to $\lambda_{n+1}\in\Lambda_{n+1}$ (note that when $n=0$, the quotient $\tilde{\mathbb{V}}_{\vec{\lambda}}$ of $\mathbb{V}_{\vec{\lambda}}$ is a shift system with respect to $\lambda_1\in\Lambda_{1}$. In the case of the triplet Virasoro algebra at level $p\in\Z_{\geq 2}$, $\mathbb{V}_\lambda$ is a projective cover of its irreducible module and $\tilde{\mathbb{V}}_{\lambda}$ is an irreducible lattice VOA-module, respectively).
The reason we must consider the quotient $\tilde{H}^0(G\times_B\cdots)$ is that the functor $H^0(G\times_B-)$ does not change a $G$-module; in other words, to apply the $n$-th FT construction, we must break the $G$-module structure of the $(n-1)$-th FT construction to make a shift system ($B$-module structure) again.
Then we finally obtain the \textit{nested FT constructions}
\begin{align*}
\underbrace{H^0(G\times_B\tilde{H}^0(G\times_B\cdots\tilde{H}^0(G\times_B}_{\text{$N$-times}}\tilde{\mathbb{V}}_{\vec{\lambda}})\cdots)).
\end{align*}
Because a shift system appears at each stage of nesting, Theorem \ref{main theorem of part 1} can be applied repeatedly. 
In particular, for $\mathfrak{g}=\mathfrak{sl}_2$, the $q$-series obtained by
iteratively applying \eqref{character formula in introduction} along the nested
FT construction surprisingly almost coincides with the
$\hat{Z}$-invariant \cite{GPPV} of the corresponding negative definite Seifert
$3$-manifold, after restricting to the $h=0$ sector (conjecturally singlet-type VOA).
This suggests that such nested FT constructions yield rich examples of LCFT whose representation theories are to some extent reducible and controllable by the geometric representation theory of FT construction, providing a blueprint for the ``dictionary" between LCFTs and negative definite plumbed $3$-manifolds expected in \cite{CCFGH,CCKPS,CFGHP}, etc.

In \cite{Sug3, Sug4},
for $\mathfrak{g}=\mathfrak{sl}_2$ and any negative definite plumbed 3-manifold $Y$, the first author 
realized the nested FT constructions on an abstract abelian category $\mathcal{C}^Y$.
Here, $\mathcal{C}^Y$ is intended to be a module category of the conjectural ``Virasoro VOA'', whose existence is predicted based on the ``dictionary'' mentioned above.
Indeed, the $\hat{Z}$-invariants of $Y$ are reconstructed in the Grothendieck group $[\mathcal{C}^Y]$ of $\mathcal{C}^Y$ up to formal
substitution of certain lattice theta functions $\Theta^Y_{\vec{\lambda}}(q)$ into the representatives $[V_{\vec{\lambda}}]\in[\mathcal{C}^Y]$ of the corresponding indecomposable objects $V_{\vec{\lambda}}\in\mathcal{C}^Y$, which are intended to be the irreducible modules of the conjectural ``lattice VOA'' whose characters are $\Theta^Y_{\vec{\lambda}}(q)$. 
The Lie algebraic framework developed here --- formulating the FT construction
without presupposing VOA structure --- makes it possible to study such
conjectural 
VOAs before their full VOA-theoretic description is available.

\noindent
\textbf{Acknowledgments:}
S.S. thanks Naoki Genra and Hiroshi Yamauchi for useful comments on naming.
H.L. thanks Antun Milas for motivating questions, and to Drazen Adamovic and Naoki Genra for useful discussions.
The authors thank Tomoyuki Arakawa for useful feedback. 
The authors also thank anonymous referee for useful suggestions.
The majority of this work was done when the authors were postdocs in YMSC, Tsinghua University.

\part{Shift system}\label{part 1}
In Part \ref{part 1}, we describe the discussion of \cite{FT,Sug,Sug2} in an axiomatic (in other words, purely Lie Algebraic) manner. 
We also generalize the discussion in these papers to include non-simply laced cases.
\subsection{Preliminaries on Lie Algebras}\label{section: preliminary}
We present the basic notations and facts used throughout this paper.
Let $\g=\mathfrak{n}_{-}\oplus\h\oplus\mathfrak{n}_{+}$ be a finite dimensional simple Lie algebra of rank $r$ and lacing number $r^\vee$, and its triangular decomposition.
Let $\Delta=\Delta_+\sqcup\Delta_{-}$ be the root system of $\g$ (resp. positive roots and negative roots).
For the Weyl group $W$ of $\g$, we consider the canonically normalized $W$-invariant bilinear form $(\cdot,\cdot)=(\cdot,\cdot)_Q$, namely, for a long root $\alpha\in\Delta^l$, the length $|\alpha|^2$ is always $2$ (denote $\Delta^s$ the short roots).
For a root $\alpha\in\Delta$, the coroot $\alpha^\vee$ is defined by $\tfrac{2}{|\alpha|^2}\alpha$.
For $1\leq i\leq r$, $\alpha_i$, $\alpha_i^\vee$, $\alpha_i^\ast$, $\varpi_i$ and $\sigma_i\in W$ are the simple root, simple coroot, fundamental coweight (dual of $\alpha_i$), fundamental weight (dual of $\alpha_i^\vee$), and the simple reflection $\sigma_i(\mu)=\mu-(\mu,\alpha_i^\vee)\alpha_i$ corresponding to $\alpha_i$, respectively.
In this paper, the labeling of the Dynkin diagrams of $\g$ is given by
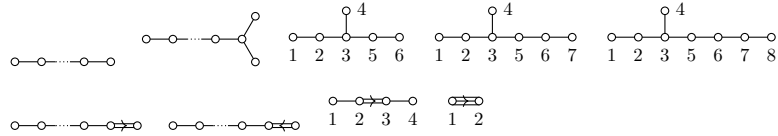
\begin{align*}
&
\begin{tikzpicture}
\dynkin[mark=o]{A}{}
\end{tikzpicture}
\ \ 
\begin{tikzpicture}
\dynkin[mark=o]{D}{}
\end{tikzpicture}
\ \ 
\begin{tikzpicture}
\dynkin[mark=o]{E}{6} 
\dynkinLabelRoot{1}{1}
\dynkinLabelRoot{2}{4} 
\dynkinLabelRoot{3}{2}
\dynkinLabelRoot{4}{3}
\dynkinLabelRoot{5}{5}
\dynkinLabelRoot{6}{6}
\end{tikzpicture}
\ \ 
\begin{tikzpicture}
\dynkin[mark=o]{E}{7} 
\dynkinLabelRoot{1}{1}
\dynkinLabelRoot{2}{4} 
\dynkinLabelRoot{3}{2}
\dynkinLabelRoot{4}{3}
\dynkinLabelRoot{5}{5}
\dynkinLabelRoot{6}{6}
\dynkinLabelRoot{7}{7}
\end{tikzpicture}
\ \ 
\begin{tikzpicture}
\dynkin[mark=o]{E}{8} 
\dynkinLabelRoot{1}{1}
\dynkinLabelRoot{2}{4} 
\dynkinLabelRoot{3}{2}
\dynkinLabelRoot{4}{3}
\dynkinLabelRoot{5}{5}
\dynkinLabelRoot{6}{6}
\dynkinLabelRoot{7}{7}
\dynkinLabelRoot{8}{8}
\end{tikzpicture}
\\
&\begin{tikzpicture}
\dynkin[mark=o]{B}{}
\end{tikzpicture}
\ \ 
\begin{tikzpicture}
\dynkin[mark=o]{C}{}
\end{tikzpicture}
\ \ 
\begin{tikzpicture}
\dynkin[mark=o,label]{F}{4}
\end{tikzpicture}
\ \ 
\begin{tikzpicture}
\dynkin[mark=o]{G}{2}
\dynkinLabelRoot{1}{1}
\dynkinLabelRoot{2}{2}
\end{tikzpicture}
\end{align*}
respectively.
We sometimes identify the set of simple roots $\Pi=\{\alpha_1,\dots,\alpha_r\}$ with
the set $I=\{1,\dots,r\}$.
The fundamental weight $\varpi_i$ is the dual vector to the coroot.
Denote $Q=\sum_{i\in I}\Z\alpha_i$, $Q^\vee=\sum_{i\in I}\Z\alpha_i^\vee$, $Q^\ast=\check{P}=\sum_{i\in I}\Z\alpha_i^\ast$ and $P=\sum_{i\in I}\Z\varpi_i$ the root lattice, coroot lattice, coweight lattice and weight lattice, respectively.
The set of dominant integral weights $P_+$ and dominant integral coweights $\check{P}_+=Q^\ast_+$ are given by $P_+=\sum_{i\in I}\Z_{\geq 0}\varpi_i$,
$\check{P}=\sum_{i\in I}\Z_{\geq 0}\alpha_i^\ast$, respectively.
We use the letter $P_{\mathrm{min}}\subseteq P_+$ for the family of minuscule weights.
Let $G$ be the corresponding simply-connected simple algebraic group with $B$ the Borel subgroup. 
Given a subset $J\subset I$, we denote by $P_J$ the corresponding parabolic subgroup\footnote{Throughout the paper, we use the negative parabolic/Borel subalgebras.} and $j\in J$, by $SL_2^j$ the corresponding subgroup isomorphic to $SL_2$. 
In the case $J=\{j\}$, we denote $P_J$
as $P_j$.
Denote $\rho=\sum_{i\in I}\varpi_i$ and $\rho^\vee=\sum_{i\in I}\alpha_i^\ast$ the Weyl vector and Weyl covector, $\theta$ and $\theta_s$ the highest (long) root and highest short root,
$h$ and $h^\vee$ the Coxeter and dual Coxeter number, respectively.
For $\sigma\in W$, $l(\sigma)$ and $w_0$ denote the length of $\sigma$ and the longest element in $W$, respectively.
For a minimal expression $\sigma=\sigma_{i_n}\cdots\sigma_{i_1}$ of $\sigma\in W$, we sometimes use $\sigma_{i_0}=\operatorname{id}$ for convenience.
For $\beta\in P_+$, denote $L_\beta$ the finite-dimensional irreducible $\g$-module with the highest weight $\beta$.
We sometimes use the letters $y_\beta$ and $y'_{\beta}$ for a highest-weight vector and lowest-weight vector of $L_\beta$, respectively.
The Langlands dual ${}^L\g$ is defined by replacing the root system by the coroot system.
Namely,
\begin{align*}
&(Q_{{}^L\g},(\cdot,\cdot)_{{}^LQ})
=
(Q^\vee,\tfrac{1}{r^\vee}(\cdot,\cdot)_Q),\\
&\alpha_{i,{}^L\g}=\alpha_i^\vee,\ \ 
\alpha_{i,{}^L\g}^\vee=r^\vee\alpha_{i},\ \ 
\alpha_{i,{}^L\g}^\ast=r^\vee\varpi_i,\ \ 
\varpi_{i,{}^L\g}=\alpha_i^\ast
\end{align*}
(in particular, 
$\rho_{{}^L\g}=\rho^\vee$ and 
$\rho_{{}^L\g}^\vee=r^\vee\rho$).
If we want to emphasize that a certain object $X$ (e.g., $\g$, $Q$,...) is derived from ${}^L\g$ rather than $\g$, we often use symbols such as ${}^LX$, $X_{{}^L\g}$, $\check{X}$, etc.
By abuse of notation, for $\theta=\sum_{i\in I}a_i\alpha_i$, denote  ${}^L\theta=\sum_{i\in I}a_i\alpha_{r+1-i}^\vee\in\h^\ast$ (i.e., we regard ${}^L\theta_{\g}=\theta_{{}^L\g}$ as an element in $\h^\ast$).
For more detailed data in each case, see e.g. \cite[p.91-92]{Kac2} (but note that the labeling of Dynkin diagrams is different).

For a weight $B$-module $M$ and $\beta\in\h^\ast$, denote $M^{h=\beta}$ the weight space with Cartan weight $\beta$.
For $\mu\in\h^\ast$, let $\C_\mu$ be the one-dimensional $B$-module such that $\mathfrak{n}_{-}$-action is trivial. 
For a $B$-module $V$ and $\mu\in\h^\ast$, we write $V(\mu)$ for the $B$-module $V\otimes_\C\C_\mu$.
We use the same notation for a vector bundle and its sheaf of sections. 
For a sheaf $\mathcal{F}$ over a topological space $X$ and an open subset $U$ of $X$, $\mathcal{F}(U)$ denotes the space of sections of $\mathcal{F}(U)$ on $U$. 
For an algebraic variety $X$, let $\mathcal{O}_X$ be the structure sheaf of $X$. 
For $\mu\in\h^\ast$, we write $\mathcal{O}(\mu)$ for the line bundle (or invertible sheaf) $G\times_B\C_\mu$ over the flag variety $G/B$. 
In particular, $\mathcal{O}(0)$ is $\mathcal{O}_{G/B}$. 
For an $\mathcal{O}_{G/B}$-module $\mathcal{F}$, we use the letter $\mathcal{F}(\mu)$ for $\mathcal{F}\otimes_{\mathcal{O}_{G/B}}\mathcal{O}(\mu)$. 
In particular, for a $B$-module $V$ and a homogeneous vector bundle $G\times_BV$, we have $(G\times_BV)(\mu)=G\times_BV(\mu)$.
For $n\in\mathbb{Z}$, the $n$-th sheaf cohomology of $\mathcal{F}$ is denoted by $H^n(\mathcal{F})$.

For later convenience, we will list several facts. 
\begin{lemma}\label{Sug 114 115 lemma 4.5}
\cite[(114), (115), Lemma 4.5]{Sug}
For a $P_i$-module $M$ and $\mu\in P$, we have
\begin{align*}
&H^n(P_i\times_BM(\mu))\simeq H^n(P_i\times_B\C_\mu)\otimes M,\\
&\dim_\C H^n(P_i\times_B\C_\mu)
=
\begin{cases}
(\mu,\alpha_i^\vee)+1&(n=0,\ (\mu,\alpha_i^\vee)\geq 0),\\
-(\mu,\alpha_i^\vee)-1&(n=1,\ (\mu,\alpha_i^\vee)<0),\\
0&(\text{otherwise}).
\end{cases}
\end{align*}
In particular, if $(\mu+\rho,\alpha_i^\vee)\geq 0$, then $H^0(P_i\times_B\C_\mu)\simeq H^1(P_i\times_B\C_{\sigma_i(\mu+\rho)-\rho})$ as $P_i$-submodules.
\end{lemma}

\begin{lemma}\label{Sug-4.2}
If $M$ is not only a $B$-module but also an $SL_2^j$-module, then $M$ has a structure of $P_j$-module. 
\end{lemma}

\proof
The isomorphism of varieties $P_j/B\simeq SL_2^j/B^j(\simeq \mathbf{P}^1)$ actually extends to an isomorphism of equivariant bundles 
$P_j\times_B M\simeq SL_2^j\times_{B^j}M$
over $\mathbf{P}^1$ and thus induces an isomorphism of the global sections
\begin{align*}
H^0(P_j\times_B M)\simeq H^0(SL_2^j\times_{B^j}M)
\overset{\text{Lemma \ref{Sug 114 115 lemma 4.5}}}{\simeq} M.
\end{align*}
Since the left-hand side implies that it is naturally a $P_j$-module, we obtain the assertion.
\endproof

\begin{lemma}\label{Sug-4.4.1}
Let $J\subset I$ and $M_j$ ($j\in J$) be  $P_j$-submodules of a 
$B$-module $M.$
If the intersection 
$$\bigcap_{j\in J}M_j\subset M$$
is closed under $e_j$ ($j\in J$), then it is a $P_J$-submodule of $M$.
\end{lemma}
\proof
It suffices to show $\ad_{e_i}^{1-c_{ij}}(e_j)=0$ for $i\neq j$ on $\bigcap_{j\in J}M_j$.
Since $[e_i,f_j]=0$, 
we have 
\begin{align*}
[\ad_{e_i}^{1-c_{ij}}(e_j),f_i]=\sum_{m=0}^{-c_{ij}}\ad_{e_i}^{-c_{ij}-m}\ad_{h_i}\ad_{e_i}^m(e_j)=\sum_{m=0}^{-c_{ij}}h_i(m \alpha_i-\alpha_j)\ad_{e_i}^{-c_{ij}}(e_j)=0
\end{align*}
Then, it suffices to show $\ad_{e_i}^{1-c_{ij}}(e_j)y^i_\alpha=0$ for any highest weight vector $y^i_\alpha$ of $\bigcap_{j\in J}M_j$ with respect to the $SL_2^{i}$-action (where $\alpha\in P$ such that $(\alpha,\alpha_i^\vee)\geq 0$).
Since $h_i((1-c_{ij})\alpha_i+\alpha_j+\alpha)=2-c_{ij}+h_i(\alpha)$, it suffices to show $f_i^{2-c_{ij}+h_i(\alpha)}\ad_{e_i}^{1-c_{ij}}(e_j)y^i_\alpha=0$.
Note that $f_i^{2-c_{ij}+h_i(\alpha)}e_i^{1-c_{ij}}=Xf_i^{h_i(\alpha)+1}$ for some $X\in U(\sll_2^i)$ by weight consideration. Then
\begin{align*}
f_i^{2-c_{ij}+h_i(\alpha)}\ad_{e_i}^{1-c_{ij}}(e_j)y^i_\alpha
=f_i^{2-c_{ij}+h_i(\alpha)}e_i^{1-c_{ij}}e_jy^i_\alpha
=Xf_i^{s+1}e_jy^i_\alpha
=X e_j f_i^{s+1}y^i_\alpha=0.
\end{align*}
This completes the proof.
\endproof

\subsection{Shift system}
In this subsection, we introduce a new concept, named \textit{shift system},
which is the main tool in the present paper, and gives its basic examples and properties.
The name ``shift system'' was chosen because, as seen in Theorem \ref{main theorem of part 1}, the shift $\sigma\uparrow\lambda$ plays a central role in our representation theory.
\begin{definition}\label{Def FT data}
A \textit{shift system} refers to a triple
\begin{align*}
(\Lambda,\uparrow,\{V_\lambda\}_{\lambda\in\Lambda})    
\end{align*}
that satisfies the following conditions:
\begin{enumerate}
\item\label{FT data 1}
$\Lambda$ is a $W$-module.
Denote $\ast\colon W\times\Lambda\rightarrow\Lambda$ the $W$-action on $\Lambda$.
\item\label{FT data 2}
$\uparrow\colon W\times\Lambda\rightarrow P$ is a map (\textit{shift map}) such that
\begin{enumerate}
\item\label{associativity of carry-over}
$\sigma_i\sigma\uparrow\lambda=\sigma_i(\sigma\uparrow\lambda)+\sigma_i\uparrow(\sigma\ast\lambda)$.
\item\label{length and carry-over}
If $l(\sigma_i\sigma)=l(\sigma)+1$, then $(\sigma\uparrow\lambda,\alpha_i^\vee)\geq 0$.
\item\label{-1 property of carry-over}
If $\lambda\not\in\Lambda^{\sigma_i}$, then $(\sigma_i\uparrow\lambda,\alpha_i^\vee)=-1$.
If $\lambda\in\Lambda^{\sigma_i}$, then $\sigma_i\uparrow\lambda=-\alpha_i$.
\end{enumerate}
Here, for $\sigma\in W$,
$\Lambda^\sigma:=\{\lambda\in\Lambda\ |\ \sigma\ast\lambda=\lambda\}$.
\item\label{FT data 3}
$V_\lambda$ is a weight $B$-module such that 
\begin{enumerate}
\item\label{conformal grading of FT data}
$V_\lambda=\bigoplus_{\Delta}V_{\lambda,\Delta}$ and each $V_{\lambda,\Delta}$ is a finite-dimensional weight $B$-submodule.
\item\label{Felder complex of FT data}
For any $i\in I$ and $\lambda\not\in\Lambda^{\sigma_i}$, there exists $P_i$-submodules $W_{i,\lambda}\subseteq V_{\lambda}$, $W_{i,\sigma_i\ast\lambda}\subseteq V_{\sigma_i\ast\lambda}$ and a $B$-module homomorphism 
$Q_{i,\lambda}\colon V_\lambda\rightarrow V_{\sigma_i\ast\lambda}(\cryov{\sigma_i}{\lambda})$ such that we have the short exact sequence 
\begin{align*}
0\rightarrow W_{i,\lambda}
\rightarrow V_{\lambda}
\xrightarrow{Q_{i,\lambda}} W_{i,\sigma_i\ast\lambda}(\sigma_i\uparrow\lambda)
\rightarrow 0
\end{align*}
of $B$-modules.
If $\lambda\in\Lambda^{\sigma_i}$, then $V_\lambda$ has the $P_i$-module structure (we sometimes write $W_{i,\lambda}=V_\lambda$).
\end{enumerate}
\end{enumerate}
Let us introduce some terms for convenience.
The action of $f_i\in\mathfrak{n}_{-}$ on $V_\lambda$ and the $B$-module homomorphism $Q_{i,\lambda}$ above are referred to as a \textit{long screening operator} and a \textit{short screening operator}, respectively.
If some object $X$ is isomorphic (resp. conjectured to be isomorphic) to $H^0(P_J\times_BV_\lambda)$ for some $V_\lambda$ and a parabolic subgroup $P_J\supseteq B$, we call $H^0(P_J\times_BV_\lambda)$ the \textit{Feigin--Tipunin construction} of $X$ (resp. FT conjecture on $X$).
\end{definition}
As mentioned in Section \ref{section: introduction}, the motivation for introducing the shift system lies in the study
of logarithmic CFTs/VOAs, particularly multiplet W-algebras (Definition \ref{def: multiplet W-algebra}). 
In the case of $\mathfrak{g}=\mathfrak{sl}_2$, the short exact sequence in Lemma \ref{lemma_felder} formed by the irreducible modules over the multiplet W-algebra (i.e., triplet Virasoro algebra) and the lattice VOA is called the \textit{Felder complex} \cite{Fel}. 
The shift system can be viewed as a purely combinatorial and Lie algebraic
formulation of this Felder complex, as well as a generalization to arbitrary $\mathfrak{g}$.

\begin{remark}\label{rmk easy facts on FT data}
\begin{enumerate}
\item\label{sum of pair shift relation}
By substituting $\sigma=\operatorname{id}$ to the axiom \eqref{associativity of carry-over}, we have $\operatorname{id}\uparrow\lambda=0$.
On the other hand, substituting $\sigma=\sigma_i$ shows that $\sigma_i\uparrow\lambda+\sigma_i\uparrow(\sigma_i\ast\lambda)=-\alpha_i$ ($\lambda\not\in\Lambda^{\sigma_i}$) or $-2\alpha_i$ ($\lambda\in\Lambda^{\sigma_i}$).
By axioms \eqref{associativity of carry-over} and \eqref{length and carry-over}, if $l(\sigma_i\sigma)=l(\sigma)-1$, then $(\sigma\uparrow\lambda,\alpha_i^\vee)\in\Z_{<0}$.
\item\label{independence of minimal exapression?}
For $\sigma\in W$, let us fix a minimal expression $\sigma=\sigma_{i_m}\cdots\sigma_{i_1}$.
Repeated use of the axiom \eqref{associativity of carry-over} shows that
\begin{align*}
\sigma\uparrow\lambda&=(\sigma_{i_m}\cdots\sigma_{i_1})\uparrow\lambda\\
&=\sum_{j=1}^{m}\sigma_{i_j}\uparrow(\sigma_{i_{j-1}}\cdots\sigma_{i_0}\ast\lambda)
-\sum_{j=1}^m(\sigma_{i_{j-1}}\cdots\sigma_{i_1}\uparrow\lambda,\alpha_{i_j}^\vee)\alpha_{i_j}.
\end{align*}
By combining it with axiom \eqref{length and carry-over}, we can rewrite the condition \eqref{strong condition} as follows:
\begin{align}\label{another form of strong condition}
\text{For any $1\leq m\leq n$, $(\sigma_{i_m}\cdots\sigma_{i_1})\uparrow\lambda=\sum_{j=1}^{m}\sigma_{i_j}\uparrow(\sigma_{i_{j-1}}\cdots\sigma_{i_0}\ast\lambda)$.}
\end{align}
Note that at this point, these conditions depend on the minimal expression of $\sigma$.
In Example \ref{example ours} below, the independence of the choice of minimal expression is clear.
Note that if $w_0\uparrow\lambda$ is independent of the minimal expression of $w_0$, then we have $w_0\uparrow\lambda=-\rho$.
\item 
\label{maximality of W_i}
The $P_i$-submodule $W_{i,\lambda}$ above is the maximal $P_i$-submodule of $V_\lambda$.
In fact, by applying the long exact sequence $H^\bullet(P_i\times_B-)$ to the short exact sequence, we have $H^0(P_i\times_BV_\lambda)\simeq W_{i,\lambda}$. 
If $M$ is a $P_i$-submodule of $V_\lambda$, then by Lemma \ref{Sug 114 115 lemma 4.5}, we have the $P_i$-module isomorphism
\begin{align*}
M\simeq H^0(P_i\times_BM)\subseteq H^0(P_i\times_BV_\lambda)\simeq W_{i,\lambda},
\end{align*}
where both isomorphisms above are given by the evaluation map $\operatorname{ev}\colon s\mapsto s(\operatorname{id}_{P_i/B})$ and its inverse.
\item\label{Sug 3.14-3.16}
For $j\in I$, $\beta\in P$ such that $(\beta,\alpha_j^\vee)\geq 0$ and $v\in V_\lambda^{h=\beta}$, we have
$v=0$ (resp. $v\in W_{j,\lambda}$) iff $f_j^{(\beta,\alpha_j^\vee)}v=0$ (resp. $f_j^{(\beta,\alpha_j^\vee)+1}v=0$).
\end{enumerate}
\end{remark}

\begin{example}\label{example ours}
Let us give an example of $(\Lambda,\uparrow)$ in the shift system.
For a fixed $x\in\h^\ast_{\R}$, any $\mu\in\h^\ast_{\R}$ has the unique decomposition $\mu=-\mu^\bullet+\mu_\bullet$, where $\mu^\bullet\in P$ and $\mu_\bullet\in\h^\ast_{\R}$ such that $0<(\mu_\bullet+x,\alpha_i^\vee)\leq 1$ for any $i\in I$.
For $\sigma\in W$, set $\sigma\ast\mu=-\mu^\bullet+\sigma(\mu_\bullet+x)-x$. 
It defines a $W$-action on $\h^\ast/Q$, and let $\Lambda$ be the unique representatives of a $W$-submodule of $\h^\ast/Q$ (i.e., in the unique decomposition $\lambda=-\lambda^\bullet+\lambda_\bullet\in \Lambda$, $\lambda^\bullet\in P_{\mathrm{min}}$).
Then the $\Lambda$ and the ``carry-over of the $W$-action"
\begin{align*}
\sigma\uparrow\lambda:=\sigma\ast\lambda_\bullet-(\sigma\ast\lambda)_\bullet\in P    
\end{align*}
satisfies the axiom in Definition \ref{Def FT data} (well-definedness and the independence above are clear). 
In fact, we have
\begin{align*}
\sigma_i\sigma\uparrow\lambda
&=\sigma_i\sigma\ast\lambda_\bullet-({\sigma_i\sigma\ast\lambda})_\bullet\\
&=\sigma_i\sigma(\lambda_\bullet+x)-x-({\sigma_i\sigma\ast\lambda})_\bullet\\
&=\sigma(\lambda_\bullet+x)-x-(\sigma(\lambda_\bullet+x),\alpha_i^\vee)\alpha_i-({\sigma_i\sigma\ast\lambda})_\bullet\\
&=\sigma\ast\lambda_\bullet-({\sigma\ast\lambda})_\bullet+({\sigma\ast\lambda})_\bullet-(\sigma(\lambda_\bullet+x)-x-({\sigma\ast\lambda})_\bullet+({\sigma\ast\lambda})_\bullet+x,\alpha_i^\vee)\alpha_i-({\sigma_i\sigma\ast\lambda})_\bullet\\
&=\sigma\uparrow\lambda+({\sigma\ast\lambda})_\bullet-(\sigma\uparrow\lambda+({\sigma\ast\lambda})_\bullet+x,\alpha_i^\vee)\alpha_i-({\sigma_i\sigma\ast\lambda})_\bullet\\
&=\sigma_i(\sigma\uparrow\lambda)+\sigma_i\uparrow(\sigma\ast\lambda),
\end{align*}
and thus \eqref{associativity of carry-over} is satisfied.
Furthermore, if $l(\sigma_i\sigma)=l(\sigma)+1$, then we have
\begin{align*}
(\sigma_i\sigma\uparrow\lambda,\alpha_i^\vee)
&=(\sigma_i\sigma\ast\lambda_\bullet-({\sigma_i\sigma\ast\lambda})_\bullet,\alpha_i^\vee)\\ 
&=-(\sigma(\lambda_\bullet+x),\alpha_i^\vee)-(({\sigma_i\sigma\ast\lambda})_\bullet+x,\alpha_i^\vee)\in\Z_{<0},
\end{align*}
because $\sigma^{-1}\alpha_i\in\Delta_+$.
In particular, when $\sigma=\operatorname{id}$, we have
\begin{align*}
(\sigma_i\uparrow\lambda,\alpha_i^\vee)=-(\lambda_\bullet+x,\alpha_i^\vee)-(({\sigma_i\ast\lambda})_\bullet+x,\alpha_i^\vee)\in\{-1,-2\}
\end{align*}
and thus \eqref{-1 property of carry-over} follows from the assumption.
On the other hand, since
\begin{align*}
(\sigma_i\sigma\uparrow\lambda,\alpha_i^\vee)
=(\sigma_i(\sigma\uparrow\lambda)+\sigma_i\uparrow(\sigma\ast\lambda),\alpha_i^\vee)
=-(\sigma\uparrow\lambda,\alpha_i^\vee)+(\sigma_i\uparrow(\sigma\ast\lambda),\alpha_i^\vee),
\end{align*}
we have
$(\sigma\uparrow\lambda,\alpha_i^\vee)>(\sigma_i\uparrow(\sigma\ast\lambda),\alpha_i^\vee)$.
By the assumption and the discussion above, we can check \eqref{length and carry-over}.\\
For a shift system of this example, we use the notation
\begin{align}
\label{G-module decomposition of FT}
H^0(G\times_BV_\lambda)
\simeq
\bigoplus_{\alpha\in P_+\cap Q}L_{\alpha+\lambda^\bullet}\otimes\mathcal{W}_{-\alpha+\lambda}.
\end{align}
In Part \ref{part 2}, we will consider the cases where $\Lambda$ is the unique representative of $\tfrac{1}{p}Q^\ast/Q$ for the $W$-action defined via
$x=\tfrac{1}{p}\rho^\vee$ ($p=r^\vee m$, $m\in\Z_{\geq 1}$) and $x=\tfrac{1}{p}\rho$ ($p=2m-1$, $m\in\Z_{\geq 1}$ and $\g=B_r$), respectively. 
\end{example}

The following lemma is shown in exactly the same manner as \cite{Sug}, so the proof is omitted.
\begin{lemma}\label{Sug-3.18}
\cite[Lemma 3.17-3.18]{Sug}
Let $x\in V_\lambda^{h=\beta}$ be a nonzero vector for some $\beta\in P_+$.
If there exists a $B$-module homomorphism $\Phi\colon L_\beta\rightarrow U(\mathfrak{b})x$ that sends the highest weight vector $y_{\beta}$ of $L_\beta$ to $x$, then $\Phi$ is an isomorphism.
In particular,
$U(\mathfrak{b})x$ is a $G$-submodule of $V_\lambda$.
\end{lemma}

\subsection{Feigin--Tipunin conjecture/construction}
Let $(\Lambda,\uparrow,\{V_\lambda\}_{\lambda\in\Lambda})$ be a shift system.
In this subsection, we will give a brief necessary and sufficient condition for the Feigin-Tipunin conjecture \cite{FT,Sug}
\begin{align*}
H^0(G\times_BV_\lambda)\simeq\bigcap_{i\in I}W_{i,\lambda}
\subseteq V_\lambda
\end{align*}
to hold.
For $(i,j,\lambda)\in I\times I\times \Lambda$, we consider the following condition (see \cite[(98),(99),(100)]{Sug}):
\begin{align}\label{Sug-98}
\text{$\lambda\in\Lambda^{\sigma_j}$ or $(\cryov{\sigma_j}{\lambda},\alpha_i^\vee)=-\delta_{ij}$}.
\end{align}
For a subset $J\subseteq I$, we also consider the following condition:
\begin{align}\label{Sug-99}
\text{$(i,j,\lambda)$ satisfies \eqref{Sug-98} for any $(i,j)\in J\times J$}.
\end{align}
When $J=I$, the condition \eqref{Sug-99} is stated as follows (this is the weak condition \eqref{weak condition} in Theorem \ref{main theorem of part 1}\eqref{main theorem 1(1)}):
\begin{align}\label{Sug-100}
\text{$(i,j,\lambda)$ satisfies \eqref{Sug-98} for any $(i,j)\in I\times I$.}
\end{align}
\begin{lemma}\label{Sug-4.3}
\cite[Lemma 4.3]{Sug}
Let $i,j\in I$ and $\lambda\in\Lambda$. 
If $(i,j,\lambda)$ satisfies \eqref{Sug-98}, then $W_{i,\lambda}\cap W_{j,\lambda}$ is closed under $e_{i}$.
In particular, if $(J,\lambda)\subseteq I\times\Lambda$ satisfies \eqref{Sug-99}, then $\bigcap_{j\in J}W_{j,\lambda}$ is a $P_J$-module.
\end{lemma}
\proof
If $\lambda\in\Lambda^{\sigma_j}$ or $i=j$, then we have 
$W_{i,\lambda}\cap W_{j,\lambda}=W_{i,\lambda}$
and the assertion is clear. 
Let us assume that $i\not=j$ and $(\cryov{\sigma_j}{\lambda},\alpha_i^\vee)=0$.
It suffices to show that if $f_i^nA_\beta\in W_{i,\lambda}\cap W_{j,\lambda}$ for some $A_\beta\in(W_{i,\lambda})_{\beta}^{e_i}$, $(\beta,\alpha_i^\vee)>0$, and $0<n\leq(\beta,\alpha_i^\vee)$, 
then $A_\beta$ (and thus $f_i^{n-1}A_\beta$) is also in $W_{i,\lambda}\cap W_{j,\lambda}$.
Since $f_i^{(\beta,\alpha_i^\vee)}A_\beta\in W_{j,\lambda}$, we have 
\begin{align}
f_{i}^{(\beta,\alpha_i^\vee)}Q_{j,\lambda}A_{\beta}
=Q_{j,\lambda}f_{i}^{(\beta,\alpha_i^\vee)}A_{\beta}=0,
\end{align}
By the assumption $(\cryov{\sigma_j}{\lambda},\alpha_i^\vee)=0$ and Remark \ref{rmk easy facts on FT data}\eqref{Sug 3.14-3.16}, we have $Q_{j,\lambda}A_{\beta}=0$, that is, $A_\beta\in W_{j,\lambda}$.
The last assertion follows from Lemma \ref{Sug-4.2} and Lemma \ref{Sug-4.4.1}.
\endproof

\begin{lemma}\label{Sug-4.4.2}
\cite[Theorem 4.4]{Sug}
Only if $(J,\lambda)$ satisfies \eqref{Sug-99}, then $\bigcap_{j\in J}W_{j,\lambda}$ is a $P_J$-submodule of $V_\lambda$.     
\end{lemma}
\proof
Let us assume that $(J,\lambda)$ does not satisfy \eqref{Sug-99}. 
Then there exists a pair $(j,i)\in J\times J$ such that $i\not=j$, $\sigma_i\uparrow\lambda\not=-\alpha_i$ and $(\sigma_i\uparrow\lambda,\alpha_j^\vee)>0$.
For $\beta\in P$ such that $(\beta,\alpha_k^\vee)\geq 0$ for $k\in J$, we take nonzero vectors 
\begin{align*}
x\in\bigcap_{k\in J}
(W_{k,\sigma\ast\lambda}^{h=\beta})^{e_k},\ \ 
y=f_i^{(\beta,\alpha_i^\vee)}x,\ \ 
z=f_j^{\sigma_i(\beta),\alpha_j}y.
\end{align*}
Clearly, we have $f_iy=0$.
Furthermore, since $y\in (W_{j,\sigma_i\ast\lambda}^{h=\sigma_i(\beta)})^{e_j}$, we have $f_jz=0$.
Denote $w\in V_\lambda^{h=\beta+\sigma_i\uparrow\lambda}$ a preimage of $x$, namely,  $Q_{i,\lambda}w=x$.
Then
\begin{align}\label{assumptionX}
0
\not=
f_i^{(\beta,\alpha_i^\vee)}w
\in\bigcap_{i\not=k\in J}(W_{k,\lambda}^{h=\sigma(\beta)+\sigma_i\uparrow\lambda})^{e_k}.
\end{align}
In fact, using the Serre relation $\ad(f_k)^{1-(\alpha_k^\vee,\alpha_i)}f_i=0$ repeatedly, we have
\begin{align*}
f_k^{(\sigma_i(\beta)+\sigma_i\uparrow\lambda,\alpha_k^\vee)+1}f_i^{(\beta,\alpha_i^\vee)}w
=
f_k^{(\beta+\sigma_i\uparrow\lambda,\alpha_k^\vee)+1}(f_k^{-(\alpha_k^\vee,\alpha_i)(\beta,\alpha_i^\vee)}f_i^{(\beta,\alpha_i^\vee)})w=0.
\end{align*}
Let us define the nonzero vector $z'$ by
\begin{align*}
z'=f_j^{(\sigma_i(\beta)+\sigma_i\uparrow\lambda,\alpha_j^\vee)}f_i^{(\beta,\alpha_i^\vee)}w.
\end{align*}
Then by the assumption $(\sigma_i\uparrow\lambda,\alpha_j^\vee)>0$, we have
\begin{align*}
Q_{i,\lambda}z'
=f_j^{(\sigma_i(\beta)+\sigma_i\uparrow\lambda,\alpha_j^\vee)}Q_{i,\lambda}f_i^{(\beta,\alpha_i^\vee)}w
=f_j^{(\sigma_i(\beta)+\sigma_i\uparrow\lambda,\alpha_j^\vee)}y
=f_j^{(\sigma_i\uparrow\lambda,\alpha_j^\vee)}z=0.
\end{align*}
By \eqref{assumptionX}, we have $z'\in \bigcap_{k\in J}W_{k,\lambda}$.
On the other hand, if $\bigcap_{k\in J}W_{k,\lambda}$ is closed under $e_j$, then we have
\begin{align*}
f_i^{(\beta,\alpha_i^\vee)}w\in \bigcap_{k\in J}W_{k,\lambda},
\end{align*}
and thus the image $y$ of $Q_{i,\lambda}$ is zero. 
It contradicts the fact $y\not=0$ above, and thus $\bigcap_{k\in J}W_{k,\lambda}$ is not closed under $e_j$, and hence not a $P_J$-submodule of $V_\lambda$.
\endproof

\begin{theorem}
\cite[Theorem 4.14]{Sug}
\label{Sug 4-14}
The evaluation map
\begin{align*}
\mathrm{ev}\colon H^0(G\times_B V_\lambda)\rightarrow V_\lambda,\ \ s\mapsto s(1_{G/B})
\end{align*}
is an injective $B$-module homomorphism.
In particular, $H^0(G\times_BV_\lambda)$ sends to the maximal $G$-submodule of $V_\lambda$.
\end{theorem}

\begin{proof}
Since the proof is the same as \cite[Lemma 4.15 - 4.18]{Sug}, we omit the details.
For $\beta\in P_+$, denote $W_{\beta}$ the set of highest weight vectors of $H^0(G\times_BV_\lambda)$ with highest weight $\beta$.
Let $G/B=\bigcap_{\sigma\in W}U_\sigma$ be the Schubert open covering.
For $s\in W_\beta$, we have $s(N_+)=s(\operatorname{id})$, and thus $s|_{U_{\operatorname{id}}}=1\otimes v$ for some $v\in V_{\lambda}^{h=\beta}$.
By considering coordinate changes between $U_\sigma$'s, $\operatorname{ev}|_{W_\beta}$ is injective.
By Lemma \ref{Sug-3.18}, $\operatorname{ev}|_{L_{\beta}\otimes W_\beta}$ so is.
Now we just need to make sure that $L_\beta\otimes W_\beta$'s are linearly independent with respect to different highest weights $\beta$'s.
For a fixed $\gamma\in P$ and $\Delta$, let us take a vector $s=\sum_{\beta\in P_+}s_\beta\in H^0(G\times_BV_{\lambda,\Delta})^{h=\gamma}$ such that $\operatorname{ev}(s)=0$.
We fix a minimal highest weight $\alpha\in P_+$ in $\{\beta\in P_+\ |\ s_\beta\not=0\}$ with respect to the standard partial order $\geq$ on $\h^\ast$, i.e. $\mu\geq \mu'$ iff $\mu-\mu'\in\bigoplus_{i\in I}\Z_{\geq 0}\alpha_i$.
Since $L_{\beta_2}^{h=w_0(\beta_1)}=\{0\}$ for $\beta_1\not\leq\beta_2$, there exists $X\in U(\mathfrak{b})$ such that 
\begin{align*}
0=X\operatorname{ev}(s)=X\operatorname{ev}(s_\alpha)=\operatorname{ev}(X s_\alpha).
\end{align*}
As $\operatorname{s}|_{L_\alpha\otimes W_\alpha}$ is injective, we have $s_\alpha=0$.
By repeating this procedure, we have $s=0$ and thus $\operatorname{ev}$ is injective.
\end{proof}

Let us prove Theorem \ref{main theorem of part 1}\eqref{main theorem 1(1)} (see also \cite[Section 4.5]{Sug}).
By Theorem \ref{Sug 4-14} and Remark \ref{rmk easy facts on FT data}\eqref{maximality of W_i}, the first half of Theorem \ref{main theorem of part 1}\eqref{main theorem 1(1)} is proved.
By Lemma \ref{Sug-4.3} and Lemma \ref{Sug-4.4.2}, $\bigcap_{i\in I}W_{i,\lambda}$ has the $G$-module structure (and thus, in $H^0(G\times_BV_\lambda)$) if and only if $\lambda\in\Lambda$ satisfies the weak condition \eqref{weak condition}. 
This completes the proof.

\subsection{Borel--Weil--Bott-type theorem}
Let $(\Lambda,\uparrow,\{V\}_{\lambda\in\Lambda})$ be a shift system.
For $\mu\in P$ and $\lambda\not\in\Lambda^{\sigma_i}$, we have 
\begin{align}\label{the short exact sequence}
0\rightarrow 
W_{i,\lambda}(\mu)
\rightarrow 
V_{\lambda}(\mu)
\rightarrow
W_{i,\sigma_i\ast\lambda}(\mu+\sigma_i\uparrow\lambda)
\rightarrow 0.
\end{align}
\begin{lemma}\cite[Lemma 4.10]{Sug}\label{Sug-4.10}
For $i\in I$, $\sigma\in W$, $\lambda\in \Lambda$ such that 
$\ell(\sigma\sigma_i)=\ell(\sigma)+1$ and
$\sigma\ast\lambda\not\in\Lambda^{\sigma_i}$,
we have short exact sequences of $P_i$-modules 
\begin{align*}
0&\rightarrow 
H^0(P_i\times_B \C_{\sigma\uparrow\lambda})\otimes W_{i,\sigma\ast\lambda}\\&\rightarrow 
H^0(P_i\times_B V_{\sigma\ast\lambda}(\sigma\uparrow\lambda))\\
&\rightarrow 
H^0(P_i\times_B \C_{\sigma\uparrow\lambda+\sigma_i\uparrow(\sigma\ast\lambda)})\otimes
W_{i,\sigma_i\sigma\ast\lambda}\rightarrow 0,
\end{align*}
\begin{align*}
0&\rightarrow 
H^1(P_i\times_B \C_{\sigma_i\sigma\uparrow\lambda})\otimes W_{i,\sigma_i\sigma\ast\lambda}
\\
&\rightarrow 
H^1(P_i\times_B V_{\sigma_i\sigma\ast \lambda}(\sigma_i\sigma\uparrow\lambda))
\\
&\rightarrow 
H^1(P_i\times_B \C_{\sigma_i\sigma\uparrow\lambda+\sigma_i\uparrow(\sigma_i\sigma\ast\lambda)})\otimes W_{i,\sigma\ast\lambda}\rightarrow 0.
\end{align*}
\end{lemma}
\proof
Let us apply the cohomology functor
$H^\bullet(P_i\times_B\text{-})$ to \eqref{the short exact sequence} with 
\begin{align*}
(\lambda,\mu)\mapsto(\sigma\ast\lambda, \sigma\uparrow\lambda),\quad (\lambda,\mu)\mapsto (\sigma_i\sigma\ast\lambda, (\sigma_i\sigma)\uparrow\lambda)
\end{align*}
respectively. 
Then the assertions follow from Lemma \ref{Sug 114 115 lemma 4.5} if we have 
\begin{align*}
H^1(P_i\times_B\C_{\sigma\uparrow\lambda})=
H^0(P_i\times_B\C_{\sigma_i\sigma\uparrow\lambda+\sigma_i\uparrow(\sigma_i\sigma\ast\lambda)})=0.
\end{align*}
By Lemma \ref{Sug 114 115 lemma 4.5} again, it suffices to show that
\begin{align*}
(\sigma\uparrow\lambda,\alpha_i^\vee)\geq 0,\ 
(\sigma_i\sigma\uparrow\lambda+\sigma_i\uparrow(\sigma_i\sigma\ast\lambda),\alpha_i^\vee)<0.
\end{align*}
They immediately follow from the assumption and axioms in Definition \ref{Def FT data}.
\endproof

\begin{theorem}\cite[Theorem 4.8]{Sug}\label{Sug-4.8-4.11}
Let us assume that $V_{\lambda}$ is a $P_i$-module if $\lambda\in\Lambda^{\sigma_i}$.
For $\lambda\in \Lambda$ and $\sigma\in W$ satisfying $\ell(\sigma_i\sigma)=\ell(\sigma)+1$ and $(\sigma\uparrow\lambda,\alpha_i^\vee)=0$, we have an isomorphism
\begin{align}\label{coh isom}
H^n(G\times_B V_{\sigma\ast \lambda}(\sigma\uparrow\lambda))\simeq H^{n+1}(G\times_B V_{\sigma_i\sigma\ast \lambda}(\sigma_i\sigma\uparrow\lambda)).
\end{align}
In particular, Theorem \ref{main theorem of part 1}\eqref{main theorem 1(2)} is proved by applying this isomorphism repeatedly.
\end{theorem}
\proof
First we consider the case $\sigma\ast\lambda\not\in\Lambda^{\sigma_i}$.
By the axiom \eqref{-1 property of carry-over} and the assumption $(\sigma\uparrow\lambda,\alpha_i^\vee)=0$, we have 
$(\sigma\uparrow\lambda+\sigma_i\uparrow(\sigma\ast \lambda),\alpha_i^\vee)=-1$.
Therefore, by Lemma \ref{Sug 114 115 lemma 4.5}, we have  
\begin{align}\label{Sug-132}
H^n(P_i\times_B \C_{\sigma\uparrow\lambda})\simeq \delta_{n,0}\C_{\sigma\uparrow\lambda},\quad H^n(P_i\times_B \C_{\sigma\uparrow\lambda+\sigma_i\uparrow(\sigma\ast\lambda)})=0
\end{align}
and by applying \eqref{Sug-132} to the first short exact sequence in Lemma \ref{Sug-4.10}, we have
\begin{align}\label{Sug-133-1}
H^n(P_i\times_B V_{\sigma\ast\lambda}(\sigma\uparrow\lambda))
\simeq 
\delta_{n,0}W_{i,\sigma\ast\lambda}(\sigma\uparrow\lambda).
\end{align}
By the assumption $\sigma\ast\lambda\not\in\Lambda^{\sigma_i}$ 
(and thus, $\sigma_i\sigma\ast\lambda\not\in\Lambda^{\sigma_i}$)
and axiom \eqref{-1 property of carry-over}, we have
\begin{align}\label{pre-Sug-131}
\sigma_i\sigma\uparrow\lambda=\sigma_i\circ(\sigma\uparrow\lambda+\sigma_i\uparrow(\sigma\ast\lambda)),
~
\sigma\uparrow\lambda=\sigma_i\circ(\sigma_i\sigma\uparrow\lambda+\sigma_i\uparrow(\sigma_i\sigma\ast\lambda)).
\end{align}
By applying Lemma \ref{Sug 114 115 lemma 4.5} to \eqref{pre-Sug-131}, we obtain that
\begin{align}\label{pre-Sug-133-2}
H^n(P_i\times_B \C_{\sigma_i\sigma\uparrow\lambda})=0,
~H^n(P_i\times_B \C_{\sigma_i\sigma\uparrow\lambda+\sigma_i\uparrow(\sigma_i\sigma\ast\lambda)})\simeq \delta_{n,1} \C_{\sigma\uparrow\lambda},
\end{align}
where the first one follows from the assumption $(\sigma\uparrow\lambda,\alpha_i^\vee)=0$ and \eqref{-1 property of carry-over},
and the second one follows from  $(\sigma\uparrow\lambda,\alpha_i^\vee)=0$ and Lemma \ref{Sug 114 115 lemma 4.5}.
Hence, by combining the second short exact sequence with \eqref{pre-Sug-133-2}, we obtain
\begin{align}\label{Sug-133-2}
H^n(P_i\times_B V_{\sigma_i\sigma\ast \lambda}(\sigma_i\sigma\uparrow\lambda))\simeq \delta_{n,1}W_{i,\sigma\ast\lambda}(\sigma\uparrow\lambda).
\end{align}
By combining \eqref{Sug-133-1} and \eqref{Sug-133-2}, we obtain
\begin{align}\label{Sug-133}
H^a(P_i\times_B V_{\sigma\ast \lambda}(\sigma\uparrow\lambda))\simeq \delta_{a,0}W_{i,\sigma\ast\lambda}(\sigma\uparrow\lambda) \simeq H^b(P_i\times_B V_{\sigma_i\sigma\ast \lambda}(\sigma_i\sigma\uparrow\lambda))
\end{align}
for $(a,b)=(0,1)$, $(1,0)$. 
Now, we have the Leray spectral sequences 
\begin{align*}
&E_2^{a,b}=H^a(G \times_{P_i}H^b(P_i\times_B V_{\sigma\ast \lambda}(\sigma\uparrow\lambda)))\Rightarrow H^{a+b}(G \times_B V_{\sigma\ast \lambda}(\sigma\uparrow\lambda)),\\
&E_2^{a,b}=H^a(G \times_{P_i}H^b(P_i\times_B V_{\sigma_i\sigma\ast \lambda}(\sigma_i\sigma\uparrow\lambda)))\Rightarrow H^{a+b}(G \times_B V_{\sigma_i\sigma\ast \lambda}(\sigma_i\sigma\uparrow\lambda)),
\end{align*}
which differ by the 1-shift for $b$ by \eqref{Sug-133}.
Thus, we obtain the assertion.

Second, we consider the case where $\sigma\ast\lambda\in\Lambda^{\sigma_i}$.
By the axiom \eqref{-1 property of carry-over}, we have 
$\sigma_i\sigma\uparrow\lambda=\sigma_i\circ(\sigma\uparrow\lambda)$.
By combining Lemma \ref{Sug 114 115 lemma 4.5} with the assumption that $V_{\sigma\ast\lambda}$ (and thus, $V_{\sigma_i\sigma\ast\lambda}$) is a $P_i$-module, we have
\begin{align}\label{Sug-136}
\begin{split}
H^a(P_i\times_B V_{\sigma\ast \lambda}(\sigma\uparrow\lambda))
&\simeq H^a(P_i\times_B \C_{\sigma\uparrow\lambda})\otimes V_{\sigma\ast \lambda}\\
&\simeq H^b(P_i\times_B \C_{\sigma_i\sigma\uparrow\lambda})\otimes V_{\sigma_i\sigma\ast \lambda}\\
&\simeq H^b(P_i\times_B V_{\sigma_i\sigma\ast \lambda}(\sigma_i\sigma\uparrow\lambda))
\end{split}
\end{align}
for $(a,b)=(0,1)$, $(1,0)$. 
By applying the assumption $(\sigma\uparrow\lambda,\alpha_i^\vee)=0$ and Lemma \ref{Sug 114 115 lemma 4.5} to \eqref{Sug-136}, we have
\begin{align}\label{Sug-138}
H^1(P_i\times_B V_{\sigma\ast \lambda}(\sigma\uparrow\lambda))
\simeq H^0(P_i\times_B V_{\sigma_i\sigma\ast \lambda}(\sigma_i\sigma\uparrow\lambda))\simeq 0.
\end{align}
By combining the Leray spectral sequences above with \eqref{Sug-136} and \eqref{Sug-138}, we obtain the assertion.
\endproof

\subsection{Weyl-type character formula}
\label{subsection: Weyl-type character formula}
For a graded vector space $V=\bigoplus_{\Delta}V_\Delta$, $\dim_\C V_\Delta<\infty$, denote 
\begin{align*}
\ch_q V=\sum_{\Delta}\dim_\C V_\Delta q^\Delta
\end{align*}
the character
(where $q$ and $\Delta$ can take multiple variables, i.e. $q=(q_i)_{1\leq i\leq n}$, $\Delta=(\Delta_i)_{1\leq i\leq n}$,  $q^\Delta=\Pi_{i=1}^nq^{\Delta_i}$).
In the same manner as \cite[Section 4.4]{Sug}, by combining the cohomology vanishing\footnote{
Note that Kempf's vanishing theorem \cite[II,4]{Jan} is proved under a weaker condition than BWB theorem \cite[II,5]{Jan}. 
The same situation might hold in our case, but the verification of this is future work (see also the footnote in Section \ref{section: simplicity revisit}).
}
$H^{n>0}(G\times_BV_\lambda)=0$ and the Atiyah--Bott localization formula \cite{AB}, we obtain the Weyl-type character formula of the FT construction $H^0(G\times_BV_\lambda)$.
An application of the case $n=0$ in Theorem \ref{main theorem of part 1}\eqref{main theorem 1(2)} will be discussed in Section \ref{section: simplicity revisit}.
\begin{corollary}\label{Sug2-3-15}
For $\lambda\in \Lambda$ satisfying the condition \eqref{strong condition}, we have
\begin{align*}
\ch_qH^0(G\times_BV_\lambda)
&=
\sum_{\beta\in P_+}\dim L_\beta
\sum_{\sigma\in W}(-1)^{l(\sigma)}
\ch_qV_\lambda^{h=\sigma\circ\beta}\\
&=
\sum_{\beta\in P_+}\dim L_\beta
\sum_{\sigma\in W}(-1)^{l(\sigma)}
\ch_qV_{\sigma\ast\lambda}^{h=\beta-\sigma\uparrow\lambda}
\end{align*}
In particular, under the decomposition \eqref{G-module decomposition of FT}, we have
\begin{align}
\label{Weyl type character formula for multiplicity W}
\ch_q\mathcal{W}_{-\alpha+\lambda}
=
\sum_{\sigma\in W}(-1)^{l(\sigma)}\ch_qV_\lambda^{h=\sigma\circ(\alpha+\lambda^\bullet)}
=
\sum_{\sigma\in W}(-1)^{l(\sigma)}\ch_qV_{\sigma\ast\lambda}^{h=\alpha+\lambda^\bullet-\sigma\uparrow\lambda}
\end{align}
\end{corollary}
\begin{proof}
As $l(\sigma)=l(\sigma^{-1})$, it suffices to check that $\ch_qV_\lambda^{h=\sigma\circ\beta}
=\ch_qV_{\sigma^{-1}\ast\lambda}^{h=\beta-\sigma^{-1}\uparrow\lambda}$ for $\beta\in P$.
First let us consider the case $\sigma=\sigma_i$.
By applying the short exact sequence in Definition \ref{Def FT data}\eqref{FT data 3}, we have
\begin{align*}
\ch_qV_\lambda^{h=\sigma_{i}\circ\beta}
=
&\ch_qW_{i,\lambda}^{h=\sigma_{i}\circ\beta}
+
\ch_qW_{i,\sigma_i\ast\lambda}^{h=\sigma_{i}\circ\beta-\sigma_i\uparrow\lambda}\\
=
&\ch_qW_{i,\lambda}^{h=\sigma_i(\sigma_{i}\circ\beta)}
+
\ch_qW_{i,\sigma_i\ast\lambda}^{h=\sigma_i(\sigma_{i}\circ\beta-\sigma_i\uparrow\lambda)}\\
=
&\ch_qW_{i,\sigma_i\ast\lambda}^{h=\beta-\sigma_i\uparrow\lambda}
+
\ch_qW_{i,\lambda}^{h=\beta-\alpha_i=(\beta-\sigma_i\uparrow\lambda)-\sigma_i\uparrow(\sigma_i\ast\lambda)}
=\ch_qV_{\sigma_{i}\ast\lambda}^{h=\beta-\sigma_{i}\uparrow\lambda},
\end{align*}
where the second equality follows from the fact that $W_{i,\lambda}$ and $W_{i,\sigma_i\ast\lambda}$ are $P_i$-modules, and the third one follows from the axiom Definition
\ref{Def FT data}\eqref{FT data 2} and Remark \ref{rmk easy facts on FT data}\eqref{sum of pair shift relation}.
Let us take a minimal expression $\sigma=\sigma_{i_n}\cdots\sigma_{i_1}$.
By applying the same relations with respect to $\sigma_{i_1},\sigma_{i_2},\dots$ repeatedly, we have
\begin{align*}
\ch_qV_{\lambda}
^{h=\sigma\circ\beta=\sigma_{i_n}\circ(\sigma_{i_{n-1}}\cdots\sigma_{i_1}\circ\beta)}
=
&\ch_qV_{\sigma_{i_n}\ast\lambda}
^{h=\sigma_{i_{n-1}}\cdots\sigma_{i_1}\circ\beta-\sigma_{i_n}\uparrow\lambda}\\
=
&\ch_qV_{\sigma_{i_{n-1}}\sigma_{i_n}\ast\lambda}
^{h=\sigma_{i_{n-2}}\cdots\sigma_{i_1}\circ\beta-\sigma_{i_n}\uparrow\lambda-\sigma_{i_{n-1}}\uparrow(\sigma_{i_n}\ast\lambda)}\\
=
&\cdots\\
=
&\ch_qV_{\sigma^{-1}\ast\lambda}
^{h=\beta-\sum_{j=1}^{n}\sigma_{n+1-j}\uparrow(\sigma_{n+2-j}\cdots\sigma_{i_{n+1}}\ast\lambda)}
\end{align*}
where $\sigma_{i_{n+1}}:=\operatorname{id}$.
By Remark \ref{rmk easy facts on FT data}\eqref{independence of minimal exapression?}, the power of the most right-hand side is $\beta-\sigma^{-1}\uparrow\lambda$.
\end{proof}


\part{Application to vertex operator superalgebras}\label{part 2}
In Part \ref{part 2}, we consider several free-field algebras and check the axioms of shift system (Definition \ref{Def FT data}). 
It enables us to use the results in Part \ref{part 1} for the study of the corresponding new multiplet W-(super)algebras.

\section{Shift systems and free field algebras}
\subsection{Preliminary from VOSA}
First, let us introduce the vertex operator superalgebra and some basic notation.
Let $V$ be a superspace, i.e., a $\mathbb{Z}_{2}$-graded vector space $V=V_{\overline{0}} \oplus V_{\overline{1}}$, where $\{ {\overline{0}}, \overline{1} \}=\mathbb{Z}_2$. 
For $a\in V$, we say that the element $a$ has parity $p(a)\in \mathbb{Z}_2$ if $a\in V_{p(a)}$.
A field $a(z)$ is a formal series of the form 
\begin{align*}
a(z)=\displaystyle\sum_{n\in \mathbb{Z}}a_{(n)}z^{-n-1},\ 
\ 
a_{(n)}\in \operatorname{End}(V)
\end{align*}
such that for any $v\in V$, one has $a_{(n)}v =0$ for some $n\gg 0$.
\begin{definition}
A \textit{vertex superalgebra (VSA)} refers to a quadruple 
\begin{align*}
(V, \mathbf{1}, T, Y(-,z)),
\end{align*}
where
$V=V_{\bar{0}}\oplus V_{\bar{1}}$ is a superspace (\textit{state space}),
$\mathbf{1}\in V_{\overline{0}}$ (\textit{vacuum vector}),
$T\in\operatorname{End}(V)$ (\textit{derivation}), 
and 
\begin{align*}
Y(-,z)\in\operatorname{End}(V)[[z^{\pm 1}]],\ \ 
Y(a,z)=:a(z)=\sum_{n\in \mathbb{Z}} a_{(n)}z^{-n-1}
\end{align*}
(the \textit{state-field correspondence})
satisfying the following axioms:
\begin{itemize}
\item For any $a\in V$, $a(z)$ is a field,
\item \textit{(translation coinvariance):} $[T,a(z)]=\partial a(z)$,
\item \textit{(vacuum):} $\mathbf{1}(z)=Id_{V}$, $a(z)\mathbf{1} |_{z=0}=a,$
\item \textit{(locality):} $(z-w)^{N}a(z)b(w)=(-1)^{p(a)p(b)}(z-w)^{N}b(w)a(z)$
    for $N\gg 0$.
\end{itemize}
By abuse of notation, we simply write $V$ for the corresponding VSA.
Other representation theoretic concepts (e.g., homomorphism, module,...) are defined in a common way, and we omit them (for more detail, see \cite{FB}).
For a VSA $V$, the \textit{parity automorphism} $\iota_V\in\operatorname{Aut}(V)$ is defined by $\iota_V(a)=p(a)a$.
Note that $\iota_{V_1\otimes V_2}=\iota_{V_1}\otimes \iota_{V_2}$.
We abbreviate $\iota_V$ to $\iota$ if $V$ is clear.
\end{definition}

Note that from the locality axiom, one can derive the following commutator formula
\begin{align*}
    a_{(n)}b(z)=b(z)a_{n}+\sum_{i\geq 0}\binom{n}{i}z^{n-i}(a_{(i)}b)(z).
\end{align*}
We call $a_{(0)}$ a \textit{zero-mode}. 
By the commutator formula,
zero-mode is a derivation: 
\begin{align}\label{eqref: zero mode is derivation}
[a_{(0)},b_{(n)}]=(a_{(0)}b)_{(n)}.
\end{align}

\begin{definition}
A VSA $V$ is called a $\tfrac{1}{2}\mathbb{Z}$-graded \textit{vertex operator superalgebra (VOSA)} if it is $\frac12 \mathbb{Z}$-graded 
$V=\oplus_{n \in \frac12 \mathbb{Z}} V_n$
and has a 
\textit{conformal vector} $\omega\in V_2$ 
such that $\omega_{(0)}=T$ and the set of operators $\{L_{n}:=\omega_{(n+1)}\}_{n\in \mathbb{Z}}$ 
satisfy the relation of the Virasoro algebra
\begin{align}\label{viro1}
[L_{m},L_{n}]=(m-n)L_{m+n}+\tfrac{m^{3}-m}{12}\delta_{m+n,0}c_{V},\ \ 
[L_n,c_{V}]=0,\ \ 
(m,n\in\Z).
\end{align}
Weak $V$-modules for a vertex operator superalgebra~$V$ are defined in complete parallel with the non-super case, except that $M$ carries a $\mathbb{Z}_2$-grading compatible with the grading on~$V$ and with the module action, and all identities (in particular the module Jacobi identity) are taken in the super sense.
A \(V\)-module \(M\) is called \textit{graded} if there exist elements \(a^1,\dots,a^n\in V\) such that
$[L_0,a^i_{(0)}]=0$ for all $i=1,\dots,n,$
the operators $L_0,\ a^1_{(0)},\dots,a^n_{(0)}$
act semisimply on \(M\), and every simultaneous eigenspace for these operators is finite-dimensional. The \textit{conformal weight} of a homogeneous vector $v$ is the eigenvalue $\Delta_v$ in $L_0 v=\Delta_v v$. In particular, we let $\Delta_{a^i}=1$ and $L_1a^i=0$ for each $1\leq i\leq n$. Therefore, we have the following decomposition:
\[
  M=\bigoplus_{\Delta,x_1,\ldots,x_n} M_{\Delta,x_1,\ldots,x_n},\]
where 
 \[M_{\Delta,x_1,\ldots,x_n}
  :=\{m\in M : L_0m=\Delta m,\ a^i_{(0)}m=x_i m\ (i=1,\ldots,n)\}.
\]
Compared with the standard notation that grades $M$ by the $L_0$-eigenvalue alone, we carry the additional labels $(x_1,\ldots,x_n)$ encoding eigenvalues of the zero modes $a^i_{(0)}$.
The \textit{character} $\ch_{q,z_1,\dots,z_n}M$ (resp. \textit{supercharacter} $\sch_{q,z_1,\dots,z_n}M$) of $M$ is defined by
\begin{align*}
&\ch_{q,z_1,\dots,z_n}M=\operatorname{tr}_{M}q^{L_0-\tfrac{c}{24}}z_1^{a^1_{(0)}}\cdots z_n^{a^n_{(0)}},
\\
&(\text{resp.\ 
$\sch_{q,z_1,\dots,z_n}M
=\sum_{x\in\{0,1\}}(-1)^x
\operatorname{tr}_{M_{\bar{x}}}q^{L_0-\tfrac{c}{24}}z_1^{a^1_{(0)}}\cdots z_n^{a^n_{(0)}}$})
\end{align*}
A graded VOSA-module $M$ is a \textit{cyclic} $V$-module,
if there exists $m\in M_{\Delta,x_1,\dots,x_n}$ for some $\Delta,x_1,\dots,x_n$ such that $M_{\Delta,x_1,\dots,x_n}=\C m$ and $M=U(V)m$.
Note that a VOSA is always cyclic because $a=a_{(-1)}\mathbf{1}$.
If we do not consider $a^1,\dots,a^n$, we will call the grading \textit{conformal grading}.
After Section \ref{Section: free field algebra}, we only work with conformally-graded VOSAs and their modules.\footnote{On the other hand, for example in \cite{CNS}, we consider the grading by $(L_0,a^1_{(0)}:=h_{(0)})$ for $h\in\sll_2\subset\hat{\sll_2}$.}
\end{definition}
If we allow the exponent of $z$ in field expansion to be a rational number, then one can define the \textit{generalized V(O)SA} similarly (see \cite{DL} for more details). 
Let us introduce the following two operations that create another VOSA module from a given one.
\begin{definition}\label{def: contragredient dual and shpectral flow twist}
Let $V$ be a VO(S)A and $M$ be a graded $V$-module.
\begin{enumerate}
\item\label{def: contragredient dual}
(\textit{Contragredient dual})
The \textit{contragredient dual} $M^\ast$ of $M$ is defined as the restricted dual of $M$, i.e., $M^\ast=\bigoplus_{\Delta,x_1,\dots,x_n}M_{\Delta,x_1,\dots,x_n}^*$ with the following VO(S)A-module structure \cite{FHL}:
\begin{align*}
\langle Y_{M^\ast}(v,z)m',m\rangle
=\langle m',Y_{M}(e^{zL_1}(-z^{-2})^{L_0}v,z^{-1})m\rangle
=:\langle m',Y_M(v,z)^\dagger m\rangle,
\end{align*}
where $v\in V$, $m\in M$, $m'\in M^\ast$, and $\langle m',m\rangle=m'(m)$. 
In other words,
\begin{align*}
v_{(n)}^\dagger m
=
\sum_{\Delta\in\Z}(-1)^\Delta
\sum_{m\geq 0}(\tfrac{L_1^m}{m!}v_\Delta)_{(-n-m-2\Delta-2)}
\end{align*}
for the conformal decomposition $v=\sum_{\Delta\in\Z}v_\Delta$, $v_\Delta\in V_\Delta$.
Direct calculation shows that  
\begin{align}
\langle m',L_nm\rangle=\langle L_{-n}m',m\rangle,
\ \ 
(a^i_{(0)})^\dagger=a^i_{(0)}.
\end{align}

\item\label{def: spectral flow twist}
(\textit{Twisted module and spectral flow twist})
For $g\in\operatorname{Aut}(V)$ with finite order $s\in\Z_{\geq 1}$,
we can define $g$-twisted modules of $V$,
which play an important role in orbifold conformal theory (see \cite{Li-tw} for more details). 
Let us give a procedure called \textit{spectral flow twist} which deforms a given twisted $V$-module.
Suppose that there exists $h\in V_{\overline{0}}$, called \textit{simple current element}, and some $k\in\C$ such that 
\begin{align*}
L_nh=\delta_{n,0}h\ (n\geq 0),
\quad 
h_{(n)}h=k\delta_{n,1}\mathbf{1}, 
\quad 
\mathrm{Spec}(h_{(0)})\in \tfrac{1}{s}\Z,
\end{align*}
where $h_{(0)}$ acts semisimplely on $V$ and $\mathrm{Spec}(h_{(0)})$ refers to the set of eigenvalues of $h_{(0)}$. 
Then one can define \textit{Li's Delta operator} \cite{Li-ssel}, 
\begin{align*}
\Delta(z):=z^{h_{(0)}}\exp{\sum_{n=1}^{\infty}\tfrac{h_{(n)}}{-n}(-z)^{-n}}.
\end{align*}
Let $M$ be a $g$-twisted $V$-module. 
Then $(M,Y_M(\Delta(z)\cdot,z))$ is a $gg_h$-twisted $V$-module, called \textit{spectral flow twist} of $M$ and denote $\mathcal{S}_h(M)$, where $g_h=\exp(2\pi ih_{(0)})$ is an automorphism of $V$ with order $s$ (see \cite[Proposition 5.4]{Li-tw}). 
Note that if $M$ is irreducible, then so is $\mathcal{S}_h(M)$.
\end{enumerate}
\end{definition}

\subsection{Simplicity theorem}
\label{section: simplicity revisit}
This subsection will not be used until Section \ref{section: main results} and may be skipped on a first reading.
In this subsection, we formalize the discussion in \cite{Sug2} to enhance the perspective and flexibility of the discussion.\footnote{
However, the authors expect that there should be a more concise way leading to the simplicity theorem in our case (ideally, they would like to include this subsection in Part \ref{part 1}).
One reason for this belief is that Theorem \ref{main theorem of part 1}\eqref{main theorem 1(2)} has strong similarity to \cite[II,5]{Jan}, but in \cite[II,5]{Jan} the simplicity theorem is proved as a direct corollary of the BWB theorem (i.e., no need to go through Kazhdan--Lusztig polynomials, etc.). 
It is future work to study and develop the theory of shift system independently of VOA.
}
Throughout this subsection, $(\Lambda,\uparrow,\{V_\lambda\}_{\lambda\in\Lambda})$ is a shift system.
\begin{lemma}\cite[Lemma 2.27, Corollary 3.3]{Sug2}\label{Sug2-3-2}
For $\lambda,\lambda'\in\Lambda$ satisfying the condition \eqref{strong condition}, if 
\begin{align*}
V_\lambda^\ast\simeq V_{w_0\ast\lambda'}(-w_0(w_0\uparrow\lambda')),\ 
w_0\uparrow\lambda-w_0(w_0\uparrow\lambda')=-2\rho
\end{align*}
(where $V_\lambda^\ast$ is the contragredient dual of $V_\lambda$ with the contragredient $B$-action),
then we have a natural isomorphism
\begin{align*}
H^0(G\times_BV_\lambda)
\simeq
H^0(G\times_BV_{\lambda'})^\ast.
\end{align*}
In particular, if $\lambda=\lambda'$, then $H^0(G\times_BV_\lambda)$ is self-dual.
\end{lemma}
\proof
By Theorem \ref{main theorem of part 1}\eqref{main theorem 1(1)} and the Serre duality, we have
\begin{align*}
H^0(G\times_BV_\lambda)
\simeq
H^{l(w_0)}(G\times_BV_{w_0\ast\lambda}(w_0\uparrow\lambda))
&\simeq
H^{l(w_0)}(G\times_BV_{\lambda'}^\ast(-2\rho)))\\
&\simeq
H^0(G\times_BV_{\lambda'})^\ast
\end{align*}
as $G$-modules.
For the naturality discussion, see \cite[Section 2]{Sug2}.
\endproof
\begin{remark}
In \cite[Proposition 6.12]{CNS}, the isomorphisms in Lemma \ref{Sug2-3-2} hold up to the spectral flow twist $\mathcal{S}^2$.
However, since $\mathcal{S}^2$ commutes with the $B$-action, it does not affect the discussion and we omit it. 
\end{remark}

In Section \ref{section: simplicity theorem under the strong condition}, we prove the simplicity theorem for the multiplet W-(super)algebras and W-(super)algebras by comparing the Weyl-type character formula  \eqref{Weyl type character formula for multiplicity W} with the ``Kazhdan--Lusztig-type" character formula described by the affine Weyl group (see Lemma \ref{uniqueness of KL} and \ref{KL character formula for twisted A_2n}). 
The following lemma summarizes that argument.
\begin{lemma}
\label{new simplicity theorem}
Let $(\Lambda, \uparrow,\{V_\lambda\}_{\lambda\in\Lambda})$ be a shift system.
For a fixed $\beta\in P_+$ and $\lambda_0,\lambda_1\in\Lambda$ such that the character $\mathrm{ch}_q\mathcal{W}_{-\beta+\lambda_i}$ of $\mathcal{W}_{-\beta+\lambda_i}$ is given by \eqref{Weyl type character formula for multiplicity W}, let us assume that there exists a family of graded vector spaces 
\begin{align*}
\{\mathbb{M}(y,\mu_{\lambda_i})
=
\bigoplus_{\Delta}\mathbb{M}(y,\mu_{\lambda_i})_\Delta\}_{y\in\hat{W}},\ \ 
\mathbb{L}(-\beta+\lambda_i)
=
\bigoplus_{\Delta}\mathbb{L}(-\beta+\lambda_i)_\Delta\ \ (i=0,1)
\end{align*}
(where $\hat{W}$ is the affine Weyl group of $\hat{\g}$, $\mu_{\lambda_i}$ is an element in $\hat{\h}^\ast$ and denote $y\sim y'$ if $y'\in Wy$)
such that
\begin{enumerate}
\item
\label{Verma coincidence condition}
$\ch_q\mathbb{M}(y,\mu_{\lambda_i}) 
=\ch_q\mathbb{M}(y',\mu_{\lambda_i})$ iff $y\sim y'$, and $\{\ch_q\mathbb{M}(y,\mu_{\lambda_i})\}_{y\in \hat{W}/\sim}$ are linearly independent.
\item 
\label{Verma and staggered B-module}
$\ch_q\mathbb{M}(y_\sigma,\mu_{\lambda_i})
=\ch_q V_{\lambda_i}^{h=\sigma\circ\beta}$ for some $y_\sigma\in \hat{W}$ s.t. $y_{\sigma}\sim y_{\sigma'}$ iff $\sigma=\sigma'$,
\item 
\label{formal KL decomposition}
$\ch_q\mathbb{L}(-\beta+\lambda_i)=\sum_{y\in\hat{W}}a_{\beta,y}\ch_q\mathbb{M}(y,\mu_{\lambda_i})$ for some $a_{\beta,y}\in\C$ (note that $a_{\beta,y}$ is independent of $\lambda_i$),
\item 
\label{simplicity at lambda0}
$\ch_q\mathcal{W}_{-\beta+\lambda_0}=\ch_q\mathbb{L}(-\beta+\lambda_0)$.
\end{enumerate}
Then we have $\ch_q\mathcal{W}_{-\beta+\lambda_1}=\ch_q\mathbb{L}(-\beta+\lambda_1)$.
\end{lemma}
\begin{proof}
By assumption, we have
\begin{align*}
\sum_{\sigma\in W}(-1)^{l(\sigma)}\ch_q V_{\lambda_0}^{h=\sigma\circ\beta}
\overset{\eqref{Weyl type character formula for multiplicity W}}{=}
\ch_q\mathcal{W}_{-\beta+\lambda_0}
\overset{\eqref{simplicity at lambda0}}{=}
\ch_q\mathbb{L}(-\beta+\lambda_0)
\overset{\eqref{formal KL decomposition}}{=}
\sum_{y\in\hat{W}}a_{\beta,y}\ch_q\mathbb{M}(y,\mu_{\lambda_0}).
\end{align*}
By applying \eqref{Verma coincidence condition} and \eqref{Verma and staggered B-module} to the equation, we have
\begin{align*}
\sum_{y'\sim y}a_{\beta,y'}
=
\begin{cases}
(-1)^{l(\sigma)}&y\sim y_\sigma,\\
0&\text{otherwise}.
\end{cases}
\end{align*}
By applying this relation to $\ch_q\mathcal{W}_{-\beta+\lambda_1}$, the assertion is proved.
\end{proof}
In the remainder of this subsection, we
consider the setting of Example \ref{example ours} under the following conditions:
\begin{enumerate}
\item\label{assumption: global VOA structure for Q}
$W_{Q}:=H^0(G\times_BV_0)$ is a generalized cyclic VOSA and $W_{Q+\lambda}:=H^0(G\times_BV_\lambda)$ is a graded $W_Q$-module for each $\lambda\in\Lambda$.
Furthermore, the natural $G$-action is in $\operatorname{Aut}(W_Q)$.
\item\label{assumption: global VOA structure for P}
$W_{P}:=\bigoplus_{\substack{\mu\in\Lambda\\\mu_\bullet=0}}H^0(G\times_BV_\mu)$ is a generalized cyclic VOSA and $W_{P+\lambda_\bullet}:=\bigoplus_{\substack{\mu\in\Lambda\\\mu_\bullet=\lambda_\bullet}}H^0(G\times_BV_\mu)$ is a graded $W_P$-module for each $\lambda\in\Lambda$.
Furthermore, the natural $G$-action is in $\operatorname{Aut}(W_P)$ 
and $W_P$ is an abelian intertwining algebra with abelian group $P/Q$ (see \cite{McR2}).
\end{enumerate} 
If \eqref{assumption: global VOA structure for P} holds, then so does \eqref{assumption: global VOA structure for Q} (but the converse does not in general).
If \eqref{assumption: global VOA structure for Q} (resp. \eqref{assumption: global VOA structure for P}) holds, then $\mathcal{W}_0\simeq W_Q^G\simeq W_P^G$ is a vertex operator super subalgebra of $W_Q$ and $\mathcal{W}_{-\beta}$ is a $\mathcal{W}_0$-module for any $\beta\in P_+\cap Q$ (resp. for any $\beta\in P_+$) and $\lambda\in\Lambda$.
In the same manner as \cite[Section 3.2]{Sug2}, we obtain the following.
\begin{lemma}
\label{lemma: general simplicity theorem for lambda is zero}
Let us assume that $\lambda_\bullet=0$ satisfies \eqref{strong condition} and \eqref{assumption: global VOA structure for Q} (resp. \eqref{assumption: global VOA structure for P}) holds.
Then $W_Q$ (resp. $W_P$) is a simple VOSA (resp. simple generalized VOSA).
Furthermore, for each $\beta\in P_+\cap Q$ (resp. $\beta\in P_+$), $\mathcal{W}_{-\beta}$ is a simple $\mathcal{W}_0$-module.
If \eqref{assumption: global VOA structure for P} holds, then $W_{Q+\lambda^\bullet}$ is a simple $W_Q$-module.
\end{lemma}
\begin{proof}
We will omit the details of the proof (see \cite[Section 3.2]{Sug2}).
Let us consider the case \eqref{assumption: global VOA structure for P}.
By Lemma \ref{Sug2-3-2}, $W_{P}$ is self-dual.
Then for the non-degenerate $W_P$-invariant bilinear form $\langle\ ,\ \rangle\colon W_P\times W_P\rightarrow\C$, we have
\begin{align*}
1=
\langle a^\ast,a\rangle
=\langle a^\ast_{(-1)}\mathbf{1},a\rangle
=\langle \mathbf{1},(a^\ast_{(-1)})^\dagger a\rangle,
\end{align*}
where $a\in W_P$ and $a^\ast\in W_P\simeq W_P^\ast$ by abuse of notation.
Since $W_P$ is cyclic, we have $(a^\ast_{(-1)})^{\dagger}a=\mathbf{1}$
and thus $W_P$ is simple.
By \cite[Theorem 3.2]{McR2}, $\mathcal{W}_{-\beta}$ is simple as $\mathcal{W}_0$-module for any $\beta\in P_+$.
Finally, by \cite[Proposition 2.26]{McR2}, each $W_{Q+\lambda^\bullet}$ is simple as $W_Q$-module.
\end{proof}

In the same manner as \cite[Lemma 3.5, 3.22]{Sug2} or \cite[Theorem 6.15]{CNS}, we have the following.

\begin{lemma}
\label{lemma: general simplicity theorem for lambda is nonzero}
Let us assume that \eqref{assumption: global VOA structure for Q} holds and for some $\lambda\in\Lambda$ and any $\alpha\in P_+\cap Q$, $\mathcal{W}_{-\alpha+\lambda_\bullet}$ is cyclic as $\mathcal{W}_0$-module 
with respect to a certain vector
$|\lambda_\bullet\rangle\in\mathcal{W}_{-\alpha+\lambda_\bullet}$.
Then $W_{Q+\lambda_\bullet}$ is simple as $W_Q$-module.
Furthermore, if \eqref{assumption: global VOA structure for P} holds and $\mathcal{W}_{-\alpha+\lambda}$ is cyclic as $\mathcal{W}_0$-module 
with respect to a certain vector
$|\lambda\rangle\in\mathcal{W}_{-\alpha+\lambda}$, then $W_{Q+\lambda}$ is simple as $W_Q$-module.
\end{lemma}
\begin{proof}
We omit the details of the proof for the same reason as above.
Let us consider the latter half.
Using the Leibniz rule \eqref{eqref: zero mode is derivation} repeatedly, $W_{P+\lambda_\bullet}$ is generated by $|\lambda_\bullet\rangle$ as $W_P$-module.
In the same manner as Lemma \ref{lemma: general simplicity theorem for lambda is zero}, $W_{P+\lambda_\bullet}$ is simple as $W_P$-module (by the same discussion for $W_Q$ and $W_{Q+\lambda_\bullet}$, the first half of the Lemma is already proved at this point).
On the other hand, $W_{Q+\lambda}$ is generated by some element in $L_{\lambda^\bullet}\otimes\C^\times|\lambda\rangle$ as $W_Q$-module.
Since $W_{Q+\lambda}\subseteq W_{P+\lambda_\bullet}$ and  $W_{P+\lambda_\bullet}$ is simple as $W_P$-module, by \cite[Corollary 4.2]{DM}, any element of $W_{Q+\lambda}$ has the form $\sum_{n\in\Z}X^n_{(n)}|\lambda\rangle$ for some $\{X^n\}_{n\in\Z}\subset W_{P}$.
By weight consideration, we can show that $\{X^n\}_{n\in\Z}\subset W_{Q}$, and thus $W_{Q+\lambda}$ is simple as $W_Q$-module.
\end{proof}

We now explain how these results are applied to the simplicity theorems for multiplet $W$-(super)algebras.
We consider the case \eqref{assumption: global VOA structure for Q} or \eqref{assumption: global VOA structure for P}.
Then, by Lemma \ref{lemma: general simplicity theorem for lambda is zero}, $\mathcal{W}_{-\beta}$ is simple as $\mathcal{W}_0$-module.
To extend the simplicity of $\mathcal{W}_{-\beta}$ to $\mathcal{W}_{-\beta+\lambda}$ such that $\lambda\in\Lambda$ satisfies \eqref{strong condition}, 
we use
Lemma \ref{new simplicity theorem} to $\lambda_0=0$ and $\lambda_1=\lambda$.
Here, $\mathbb{M}(\mu)$ and $\mathbb{L}(\mu)$ represent the Verma module and the irreducible module over $\mathcal{W}_0$ with highest weight $\mu$, respectively.
As we show later, sometimes $\mathbb{L}(-\beta+\lambda_i)$ ($i=0,1$) has a ``Kazhdan--Lusztig-type" character formula in Lemma \ref{new simplicity theorem}\eqref{formal KL decomposition}.
Since $\mathcal{W}_{-\beta+\lambda_0}$ is irreducible, by Lemma \ref{new simplicity theorem}, so is $\mathcal{W}_{-\beta+\lambda_1}$.
In particular, $\mathcal{W}_{-\beta+\lambda}$ are cyclic.
Furthermore, if they are cyclic, by Lemma \ref{lemma: general simplicity theorem for lambda is nonzero}, $W_{Q+\lambda_\bullet}$ or $W_{Q+\lambda}$ is also simple as $W_Q$-module.

In \cite{Sug,Sug2} and \cite{CNS}, the cases where $\mathcal{W}_0$ is the principal W-algebra $\mathcal{W}_k(\g)$ for simply-laced $\g$ and the affine VOA $V_k(\sll_2)$ were considered, respectively, and the corresponding simplicity theorems were proved in the same manner explained above.
In Section \ref{section: main results} and Section \ref{section: simplicity theorem under the strong condition} of the present paper, we consider the cases where $\mathcal{W}_0$ is the principal W-(super)algebra $\mathcal{W}_k(\g)$ for non-simply-laced $\g$ and $\g=\osp(1|2r)$, respectively. 
In the former case, \eqref{assumption: global VOA structure for P} holds and the simplicity theorem will be proved for any $\lambda\in\Lambda$ satisfying \eqref{strong condition}, but in the later case, it will be proved only for $\lambda^\bullet=0$ case (i.e., assume only \eqref{assumption: global VOA structure for Q}) because of some technical difficulty.

\subsection{Free field algebras}\label{Section: free field algebra}
Let us introduce two VOSAs that will serve as material for shift systems later.
\subsubsection{Fermionic VOSA}
The Clifford algebra is generated by $\{\phi(n),n\in \frac12+\mathbb{Z}\}\cup \{{1}\}$ satisfying the relations
\begin{align*}
    \{\phi(n),\phi(m)\}=\delta_{n,-m},\quad n,m\in \tfrac12+\mathbb{Z}.
\end{align*}
Let $F$ be the module of Clifford algebra generated by $\bold{1}$ such that 
$\phi(n)\bold{1}=0$ ($n>0$).
Then there is a unique VOSA structure on $F$, with the field and conformal vector
\begin{align}\label{conformal vector of free fermion}
Y(\phi(-\tfrac12)\bold{1},z)=\phi(z)=\sum_{n\in \frac12+\mathbb{Z}}\phi(n)z^{-n-\frac12},\ \ 
\omega^{(s)}=\tfrac12 \phi(-\tfrac32)\phi(-\tfrac12)\bold{1},
\end{align}
which gives the central charge $\tfrac12$.
The conformal weight of $\Delta_\phi$ of $\phi(-\tfrac{1}{2})\mathbf{1}$ is $\tfrac{1}{2}$.
The character and supercharacter of \(F\), and the character of the
\(\iota\)-twisted \(F\)-module \(\iota^*F\), are given by
\[
\ch_q F
=
\frac{\eta(q)^2}{\eta(q^2)\eta(q^{1/2})},
\qquad
\operatorname{sch}_q F
=
\frac{\eta(q^{1/2})}{\eta(q)},
\qquad
\ch_q \iota^*F
=
2\frac{\eta(q^2)}{\eta(q)}.
\]
Here \(\iota\) denotes the parity automorphism.
Moreover, $(\phi(-\frac12)v^*,v)=(v^*,\text{i} \phi(-\frac12)v)$ for the dual $v^\ast\in F^\ast$ of $v\in F$.

\subsubsection{Lattice VOSAs}
\label{subsection lattice and fermionic}
Let $L$ be a positive definite integral lattice of rank $r$. 
We extend the bilinear form on $L$ to $\mathfrak{h}_L=\mathbb{C}\otimes_{\mathbb{Z}} L$. 
We also consider the affinization $\hat{\h}_L=\mathbb{C}[t,t^{-1}]\otimes \mathfrak{h}_L\oplus \mathbb{C}c$ of $\mathfrak{h}_L$ and its subalgebras $\hat{\mathfrak{h}}_L^{\geq 0}=\mathbb{C}[t]\otimes \mathfrak{h}_L$, $\hat{\mathfrak{h}}_L^-=t^{-1}\mathbb{C}[t^{-1}]\otimes \mathfrak{h}_L$.
For $\mu\in \mathfrak{h}$, the induced representation $M(\mu)$ of $\hat{\mathfrak{h}}_L$ is defined by
\begin{align*}
M(\mu)=U(\hat{\mathfrak{h}}_L)\otimes_{U(\hat{\mathfrak{h}}_L^{\geq 0}\oplus \mathbb{C}c)}\mathbb{C}_{\mu},
\end{align*}
where $h\in\hat{\h}_L$ acts as $(\mu,h)$ on $\mathbb{C}_{\mu}$. 
Then there exists a simple V(S)A structure (see \cite{Kac1}) on 
\begin{align*}
V_{L}
=
M(0)\otimes \mathbb{C}[L]
\simeq
\bigoplus_{\mu\in L}M(\mu)
\end{align*}
with the vacuum vector $\mathbf{1}=1\otimes 1$, such that 
$Y(ht^{-1}\otimes 1,z)=h(z)=\sum_{n\in \mathbb{Z}}h_{(n)}z^{-n-1}$, $h_{(n)}=h\otimes t^n$.
The group algebra of $\mathbb{C}[L]$ is generated by $e^\mu$, $\mu\in {L}$.
The vertex operator of $Y(e^\mu,z)$ is defined as 
\begin{align*}
Y(e^\mu,z)=e^{\mu}z^{\mu_{(0)}}\exp\left(-\sum_{j<0}\tfrac{z^{-j}}{j}\mu_{(j)}\right)\exp\left(-\sum_{j>0}\tfrac{z^{-j}}{j}\mu_{(j)}\right)\epsilon(\mu,\cdot),
\end{align*}
where $e^\mu(e^{\lambda})=e^{\mu+\lambda}$ for $\lambda\in \mathbb{C}[{L}]$, and $\epsilon(\cdot,\cdot)$ is the $2$-cocycle satisfying \cite[(5.4.14)]{Kac1}. The parity of the element $v\otimes e^{\mu}$ is $|\mu|^2\mod{2}$. 
One can define a (shifted) conformal vector
\begin{align}\label{shifted conformall vector}
\omega_{\mathrm{sh}}=\omega_{\mathrm{st}}+\gamma_{(-2)}\mathbf{1},\ \ 
\omega_{\mathrm{st}}
=\tfrac{1}{2} \sum_{i=1}^{r} v_{i (-1)}v_i^\ast,
\ \ 
\gamma\in \mathfrak{h}_L
\end{align}
with central charge is $r-12(\gamma,\gamma)$,
where $\{v^i\}_{i=1}^{\ell}$ and $\{v_i^\ast\}_{i=1}^{r}$ are bases of $L$ and $L^\ast$, respectively, such that $(v_i,v_j^\ast)=\delta_{i,j}$. 
The dual lattice $L^\ast$ of $L$ is defined by
$L^\ast:=\{\beta\in \mathbb{Q}\otimes L\;|\; (\alpha,\beta)\in \mathbb{Z}\,\,\text{for all}\,\, \alpha\in L\}$.
Then the lattice VO(S)A has finitely many irreducible modules parametrized by $L^\ast/L$ \cite{LiL}; these modules are 
\begin{align*}
V_{L+\lambda}:=M(0)\otimes e^{\lambda}\mathbb{C}[L],\ \ (\lambda\in L^\ast/L)
\end{align*}
up to isomorphism.
For $\mu\in \lambda+L$, the conformal weight $\Delta_\mu$ of $e^\mu$ under the action of $(\omega_{\mathrm{sh}})_{(1)}=L_0$ is 
\begin{align}\label{conformal weight of lattice point vector}
\Delta_{\mu}
=\tfrac{1}{2}|\mu-\gamma|^2+\tfrac{c-r}{24}
=\tfrac{1}{2}|\mu|^2-(\mu,\gamma)
.
\end{align}
According to \cite[Proposition 5.3.2]{FHL}, the contragredient dual $V_{L+\lambda}^\ast$ of $V_{L+\lambda}$ is $V_{L+\lambda'}$, for some $\lambda'\in L^\ast/L$.

\subsubsection{Ramond sector}
\label{sectionramond}
Let us consider the setup in the last two sections and Definition \ref{def: contragredient dual and shpectral flow twist}\eqref{def: spectral flow twist}.
For
\begin{align*}
R=\tfrac{1}{2}\sum_{i\in I_{\mathrm{odd}}}v_i^\ast,\ \ 
I_{\mathrm{odd}}
=
\{1\leq i\leq r\ |\ \text{$|v_i|^2$ is odd}\},
\end{align*}
we can easily check that $R$ is a simple current element and $g_R=\operatorname{exp}(2\pi iR_{(0)})$ gives the parity automorphism $\iota_{V_L}$ of $V_L$.
Let $M$ be a $1\otimes\iota_{F}$-twisted $V_L\otimes F$-module.
Note that $\iota_{V_L}\otimes 1\in\operatorname{Aut}(V_L\otimes F)$.
The \textit{Ramond sector} of $M$ is defined by the $\iota_{V_L\otimes F}$-twisted $V_L\otimes F$-module $\mathcal{S}_{R}(M)$.

\subsection{Construction of shift systems}
\label{section:construction of shift system}
Let us consider step-by-step if we can build a shift system (Definition \ref{Def FT data}) using the two VOSA(-modules), free fermion $F$ and lattice VOSA-module $V_{L+\lambda}$ with conformal vectors \eqref{conformal vector of free fermion} and \eqref{shifted conformall vector}, respectively, introduced above.
To consider $V_{L+\lambda}$ as a weight $B$-module, it is natural to consider the case $L=xQ$ for some $x\in\C^\times$ and the $\mathfrak{h}$-action
\begin{align}\label{the Cartan action for our cases}
h_i=\lceil -\tfrac{1}{x}\alpha_{i (0)}^\vee\rceil.
\end{align}
In other words, we consider the Fock space $M(-x\mu)$ as the Cartan weight space with weight $\lceil \mu\rceil$, where $\lceil \mu\rceil\in P$ is uniquely determined by $\lceil(\mu,\alpha_i^\vee)\rceil=(\lceil \mu\rceil,\alpha_i^\vee)$ for $i\in I$. 
We consider the case where $L=xQ$ is positive integral (see Section \ref{subsection lattice and fermionic}), which is equivalent to the cases $x=\sqrt{p}$ for some $p\in\Z_{\geq 1}$ such that
\begin{enumerate}
\item\label{non-super case}
$p=r^\vee m$ (where $r^\vee$ is the lacing number of $\g$ and $m\in\Z_{\geq 1}$), or
\item\label{super case}
$p=2m-1$, $m\in\Z_{\geq 1}$ and $\g=B_r$.
\end{enumerate}
Note that in the case \eqref{non-super case}, $L=\sqrt{p}Q$ is even, but in the case \eqref{super case}, $\sqrt{p}Q$ is odd (because $|\sqrt{p}\alpha_r|^2$ is odd).
Also, since $\sqrt{p}Q$ is an integral, $\sqrt{p}Q\subseteq (\sqrt{p}Q)^\ast$ and thus we can consider $(\sqrt{p}Q)^\ast/\sqrt{p}Q$. 
For the moment, we will not fix the representative $\Lambda$ of $(\sqrt{p}Q)^\ast/\sqrt{p}Q\simeq \tfrac{1}{p}Q^\ast/Q$, and $\lambda$ will simply represent an element of $\tfrac{1}{p}Q^\ast$ and we use the letter $V_{\sqrt{p}(Q+\lambda)}$ for the corresponding irreducible $V_{\sqrt{p}Q}$-module.
The representative $\Lambda$ will be determined later from considerations on the choice of conformal weight vector.

Since the $\h$-action is given by \eqref{the Cartan action for our cases}, it is natural to consider whether the zero-modes 
\begin{align*}
e^{\sqrt{p}\alpha_i}_{(0)}\colon
V_{\sqrt{p}(Q+\lambda)}^{h=\lceil\mu\rceil}=M(-\sqrt{p}\mu)
\rightarrow 
M(-\sqrt{p}(\mu-\alpha_i))=V_{\sqrt{p}(Q+\lambda)}^{h=\lceil\mu-\alpha_i\rceil=\lceil\mu\rceil-\alpha_i}
\end{align*} 
define a $B$-action of a shift system.  
For a $B$-action to be compatible with the structure of VOSA,  it should be even, but in the case \eqref{super case}, $e^{\sqrt{p}\alpha_r}_{(0)}$ is odd as we checked above.
To avoid this problem, in the case of \eqref{super case}, let us consider $V_{\sqrt{p}(Q+\lambda)}\otimes F$ and $(e^{\sqrt{p}\alpha_r}\otimes\phi)_{(0)}$ instead of $V_{\sqrt{p}Q}$ and $e^{\sqrt{p}\alpha_r}_{(0)}$, respectively.
Then since $\phi\in F$ is odd, $(e^{\sqrt{p}\alpha_r}\otimes\phi)_{(0)}$ is even.
From now on we consider the following two cases:
For case \eqref{non-super case} and \eqref{super case},
\begin{align}
&V_\lambda=V_{\sqrt{p}(Q+\lambda)},\ \ 
\omega=\omega_{\mathrm{sh}},\ \ 
h_i=\lceil-\tfrac{1}{\sqrt{p}}\alpha_{i(0)}^\vee\rceil,\ \ 
f_i=e^{\sqrt{p}\alpha_i}_{(0)}\ \ (1\leq i\leq r),
\label{setup in the non-super case}
\\
&V_\lambda=V_{\sqrt{p}(Q+\lambda)}\otimes F,\ \ 
\omega=\omega_{\mathrm{sh}}+\omega^{(s)},\ \ 
h_i=\lceil-\tfrac{1}{\sqrt{p}}\alpha_{i(0)}^\vee\rceil,\\ 
&f_i=
\begin{cases}
e^{\sqrt{p}\alpha_i}_{(0)}&(1\leq i\leq r-1),\\
(e^{\sqrt{p}\alpha_r}\otimes \phi)_{(0)}&(i=r),
\end{cases}
\label{setup in the super case}
\end{align} 
respectively.
Recall that in Definition \ref{Def FT data}\eqref{conformal grading of FT data}, we need a grading of $V_\lambda$ which decomposes $V_\lambda$ into finite-dimensional weight $B$-modules.
Here we shall employ the conformal grading by $\omega$ above.
Then, for the operators $\{f_i,h_i\}_{i\in I}$ to define the $B$-action in the shift system, in addition to the Serre relation and integrability, the conformal weights must be preserved\footnote{Note that compatibility with the entire Virasoro action is not necessarily required.}.
In each case \eqref{non-super case}, \eqref{super case}, let us determine such a conformal vector $\omega$ and describe the shift system naturally derived from the choice of $\omega$.
In particular, we will give the pair $(\Lambda,\uparrow)$ as special cases of Example \ref{example ours}, and thus Definition \ref{Def FT data}\eqref{FT data 1},\eqref{FT data 2} is satisfied.

\subsubsection{The case \eqref{non-super case}}
\label{subsubsection: construction of non-super shift system}
By \eqref{conformal weight of lattice point vector}, the operators $f_i$ ($i\in I$) preserve the conformal weight iff
\begin{align}\label{conformal vector for the case of non-super}
\omega
=&\tfrac{1}{2}\sum_{i\in I}\alpha_{i(-1)}\alpha_i^\ast+
\sqrt{p}(\rho-\tfrac{1}{p}\rho^\vee)_{(-2)}\mathbf{1}
\end{align}
and thus we consider this conformal vector.
On the other hand, under the choice \eqref{conformal vector for the case of non-super}, another solution of $\Delta_{e^{x\alpha_i}_{(0)}}=0$ is $x=-\tfrac{2}{\sqrt{p}(\alpha_i,\alpha_i)}$, namely the case $x\alpha_i=-\tfrac{1}{\sqrt{p}}\alpha_i^\vee$.
By \cite{TK}, if $(p\lambda+\rho^\vee,\alpha_i)\equiv s\ (\operatorname{mod}\ p)$ for $1\leq s\leq p-1$, then we have the non-zero (conformal weight-preserving) screening operators
\begin{align}\label{short screening operator in non-super case}
Q_{i,\lambda}=
\int_{[\Gamma]} e^{-\tfrac{1}{\sqrt{p}}\alpha_i^\vee}(z_1)\cdots e^{-\tfrac{1}{\sqrt{p}}\alpha_i^\vee}(z_s)dz_1...d_{z_{s}}
\colon
V_{\lambda}\rightarrow
V_{\lambda-\tfrac{s}{p}\alpha_i^\vee},
\end{align}
where $[\Gamma]$ is the cycle defined in \cite[Section 2.3]{NT}. 
To use \eqref{short screening operator in non-super case} as short screening operators of a shift system, we have to regard $\lambda-\tfrac{s}{p}\alpha_i^\vee$ as $\sigma_i\ast\lambda$ for some $W$-action $\ast$.
If we choose $x=\tfrac{1}{p}\rho^\vee$ in Example \ref{example ours}, then the representative $\Lambda$ of $\tfrac{1}{p}Q^\ast/Q$ is given by
\begin{align}\label{the representatives in non-super case}
\Lambda=\{
-\lambda^\bullet+\lambda_\bullet\ |\ 
\lambda^\bullet\in P_{\mathrm{min}},\ 
\lambda_{\bullet}=\sum_{i\in I}\tfrac{\lambda_i-1}{p}\alpha_i^\ast,\  
1\leq\lambda_i\leq r_i^\vee m
\},\ 
r_i^\vee=
\begin{cases}
r^\vee&(\text{$\alpha_i$ is long}),\\
1&(\text{$\alpha_i$ is short}),
\end{cases}
\end{align}
where $P_{\mathrm{min}}$ is the set of minuscule weights,
and has the $W$-action $\ast$ and the shift map $\uparrow$ induced from
\begin{align}\label{W-action in non-super case}
\sigma\ast\mu=\sigma(\mu+\tfrac{1}{p}\rho^\vee)-\tfrac{1}{p}\rho^\vee.
\end{align}
Then under the $W$-action, the short screening operator \eqref{short screening operator in non-super case} is described as
$Q_{i,\lambda}\colon
V_{\lambda}\rightarrow 
V_{\sigma_i\ast\lambda}(\sigma_i\uparrow\lambda)$.

\subsubsection{The case \eqref{super case}}
\label{subsubsection: construction of super shift system}
By \eqref{conformal weight of lattice point vector} and $\Delta_\phi=\tfrac{1}{2}$, the operators $f_i$ ($i\in I$) preserve the conformal weight iff
\begin{align}\label{conformal vector for the case of super}
\omega
=\omega^{(s)}+\tfrac{1}{2}\sum_{i\in I}\alpha_{i(-1)}\alpha_i^\ast+
\sqrt{p}(1-\tfrac{1}{p})\rho_{(-2)}\mathbf{1}
\end{align}
and thus we consider this conformal vector (note that the $\omega_{\mathrm{sh}}$-part is different from \eqref{conformal vector for the case of non-super}).
On the other hand, under the choice \eqref{conformal vector for the case of super}, another solution of $\Delta_{e^{x\alpha_i}_{(0)}}=0$ ($i\not=r$) and $\Delta_{(e^{x\alpha_r}\otimes\phi)_{(0)}}=0$ is $x=-\tfrac{1}{\sqrt{p}}$.
In the same manner as above,
if $(p\lambda+\rho,\alpha_i^\vee)\equiv s\ (\operatorname{mod}\ p)$ for $1\leq s\leq p-1$, then we have the non-zero (conformal weight-preserving) screening operators (see also \cite{IK1})
\begin{align}\label{short screening operator in super case}
Q_{i,\lambda}=
\begin{cases}
\int_{[\Gamma]} e^{-\frac{1}{\sqrt{p}}\alpha_i}(z_1)\cdots e^{-\frac{1}{\sqrt{p}}\alpha_i}(z_s)dz_1...d_{z_{s}}&(i\not=r),\\
\int_{[\Gamma]} (e^{-\frac{1}{\sqrt{p}}\alpha_r}\otimes\phi)(z_1)\cdots
(e^{-\frac{1}{\sqrt{p}}\alpha_r}\otimes\phi)(z_{s})dz&(i=r),
\end{cases}
\colon
V_{\lambda}\rightarrow
V_{\lambda-\frac{s}{p}\alpha_i}.
\end{align}
If we choose $x=\tfrac{1}{p}\rho$ in 
Example \ref{example ours}, then the representative $\Lambda$ of $\tfrac{1}{p}Q^\ast/Q$ is given by
\begin{align}\label{the representatives in super case}
\Lambda=\{
-\lambda^\bullet+\lambda_\bullet\ |\ 
\lambda^\bullet\in P_{\mathrm{min}},\ 
\lambda_{\bullet}=\sum_{i\in I}\tfrac{\lambda_i-1}{p}\varpi_i,\  
1\leq\lambda_i\leq p, \ \ 
\lambda_r+(\lambda^\bullet,\alpha_r^\vee)\equiv 1\ (\operatorname{mod}\ 2)
\},
\end{align}
and has the $W$-action $\ast$ and the shift map $\uparrow$ induced from
\begin{align}\label{W-action in super case}
\sigma\ast\mu=\sigma(\mu+\tfrac{1}{p}\rho)-\tfrac{1}{p}\rho.
\end{align}
Then, under the $W$-action, the short screening operator \eqref{short screening operator in super case} is described as
$Q_{i,\lambda}\colon
V_{\lambda}\rightarrow 
V_{\sigma_i\ast\lambda}(\sigma_i\uparrow\lambda)$.

The following lemma is proved in the same manner as \cite[Appendix]{Sug}, so we will omit the proof.
\begin{lemma}
\label{lemm: strong condition for our cases}
\cite[Appendix]{Sug}
\begin{enumerate}
\item
\label{lemm: strong condition for non-super case}
For the $W$-module $\Lambda$ defined by \eqref{the representatives in non-super case}, \eqref{W-action in non-super case}, we have
\begin{align*}
\text{$\lambda\in\Lambda$ satisfies \eqref{strong condition}}
\iff
(p\lambda_\bullet+\rho^\vee,\theta_{s,\g})\leq m
\iff
(p\lambda_\bullet+\rho^\vee,{}^L\theta)\leq p.
\end{align*}
In particular, $m\geq {}^Lh_{\g}^\vee-1$.
Furthermore, for such a $\lambda\in\Lambda$, we have $w_0\uparrow\lambda=-\rho^\vee$.
\item 
\label{lemm: strong condition for super case}
For the $W$-module $\Lambda$ defined by \eqref{the representatives in super case}, \eqref{W-action in super case}, we have
\begin{align*}
\text{$\lambda\in\Lambda$ satisfies \eqref{strong condition}}
\iff
(p\lambda_\bullet+\rho,{}^L\theta)\leq p.
\end{align*}
In particular, $m\geq r$.
Furthermore, for such a $\lambda\in\Lambda$, we have $w_0\uparrow\lambda=-\rho$.
\end{enumerate}
\end{lemma}

\subsubsection{Verification of Definition \ref{Def FT data}\eqref{FT data 3}}
Let us prove the remaining conditions of the shift system. 
Since the proofs are similar, we will treat both cases \eqref{non-super case} and \eqref{super case} simultaneously.
\begin{lemma}\label{baction}
The operators $\{f_i,h_i\}_{i\in I}$ give rise to an integrable $\mathfrak{b}$-action on $V_\lambda$.
\end{lemma}
\begin{proof}
The proof is based on the evaluation of conformal weights.
Let us check the Serre relation 
\begin{align*}
(\operatorname{ad}f_i)^{1-c_{ij}}f_j=0    
\end{align*}
($[h_{i,\lambda}, f_j]=-c_{ij}f_j$ is clear).
Set $f_i=x_{(0)}$ and $f_j=y_{(0)}$.
By \eqref{eqref: zero mode is derivation}, we have
$(\operatorname{ad}f_i)^{1-c_{ij}}f_j
=
(f_i^{1-c_{ij}}y)_{(0)}$ and $f_i^{1-c_{ij}}y\in V_\lambda^{h=-((1-c_{ij})\alpha_i+\alpha_j)}$.
Since $\Delta_{f_i}=\Delta_{f_j}=0$, we have 
$\Delta_{f_i^{1-c_{ij}}y}=\Delta_y=1$.
On the other hand, by \eqref{conformal weight of lattice point vector} and \eqref{conformal vector for the case of non-super}, the minimal conformal weight of $V_\lambda^{h=-((1-c_{ij})\alpha_i+\alpha_j)}$ is given by
\begin{align*}
\Delta_{\sqrt{p}((1-c_{ij})\alpha_i+\alpha_j)}
=
\begin{cases}
\tfrac{5}{2}&\text{the case \eqref{super case} and\ }(i,j)=(r-1,r)\ \text{or}\ (r,r-1),
\\
2-c_{ij}&\text{otherwise}.
\end{cases}
\end{align*}
Since this is greater than $1$, we have $(\operatorname{ad}f_i)^{1-c_{ij}}f_j=0$.
We show that the $\mathfrak{b}$-action is integrable. 
The case of $h_i$ is clear and we consider the case of $f_i$.
For $\beta\in P$, pick any homogeneous element $A$ in $V_\lambda^{h=\beta}$. 
Then we have $f_{i}^{N}A\in V_\lambda^{h=\beta-N\alpha_i}$ for $N\in \mathbb{N}$. 
The conformal weight of every nonzero element in $V_\lambda^{h=\beta-N\alpha_i}$ is greater than or equal to $\Delta_{\sqrt{p}(\beta+N\alpha_i)+\lambda}$. 
By \eqref{conformal weight of lattice point vector}, $\Delta_{\sqrt{p}(\beta+N\alpha_i)+\lambda}$ is a quadratic function in $N$ with leading coefficient $p(\alpha_i,\alpha_i)>0$. 
Meanwhile, the conformal weight $f_{i}^{N}A$ will stay the same since $f_i$ preserves the conformal weight. 
Thus, if $N$ is big enough such that $\Delta_{\sqrt{p}(\beta+N\alpha_i)+\lambda}>\Delta_A$, then $f_i^NA=0$.
\end{proof}

From now on, we consider the $B$-action on $V_\lambda$ in Lemma \ref{baction}.

\begin{lemma}
The short screening operators $Q_{j,\lambda}\colon V_\lambda\rightarrow V_{\sigma_j\ast\lambda}(\sigma_j\uparrow\lambda)$ are $B$-module homomorphisms.
\end{lemma}
\begin{proof}
It suffices to check that $[f_i,Q_{j,\lambda}]=0$.
First, we consider the case where $s=1$. 
By calculation, we have
$e^{\sqrt{p}\alpha_i}_{(r)}e^{-\frac{1}{\sqrt{p}}\alpha_j}=0$
for $r\geq 0$, $i\neq j$. 
Thus, we have $[f_{i},{Q}_{j,\lambda}]=0$ for $i\neq j$.  When $i\neq n$, we have
\begin{align}\label{for1}
    e^{\sqrt{p}\alpha_{i}}_{(0)}e^{-\frac{1}{\sqrt{p}}\alpha_i}=\sqrt{p}(\alpha_i)_{(-1)}e^{(\sqrt{p}-\frac{1}{\sqrt{p}})\alpha_i}\bold{1}=\tfrac{p}{p-1}(L_{-1}e^{(\sqrt{p}-\frac{1}{\sqrt{p}})\alpha_i})\bold{1}.
\end{align}
Similarly, when $i=n$, we have
\begin{align}\label{for2}
 e^{\sqrt{p}\alpha_{n}}_{(-1)}e^{-\frac{1}{\sqrt{p}}\alpha_n}=\tfrac{p}{p-1}(L_{-1}e^{(\sqrt{p}-\frac{1}{\sqrt{p}})\alpha_n})\bold{1}.
\end{align}
Using commutator formula, (\ref{for1}) and (\ref{for2}), we have 
\begin{align}\label{commutatorformu}
&[f_{i}, e^{-\frac{1}{\sqrt{p}}\alpha_j}(z)]=0,\quad \text{for}\;\; i\neq j,\\
&[f_{i},e^{-\frac{1}{\sqrt{p}}\alpha_i}(z)]=\tfrac{p}{p-1}\tfrac{d}{dz}e^{-\frac{1}{\sqrt{p}}\alpha_i}(z), \quad \text{for}\;\; i\neq n,\\
&[f_{n},(e^{-\frac{1}{\sqrt{p}}\alpha_n}\otimes \phi(-\tfrac12))(z)]=\tfrac{p}{p-1}\tfrac{d}{dz}e^{-\frac{1}{\sqrt{p}}\alpha_n}(z).
\end{align}
Note $(L_{-1}u)_{(0)}=0$ in every generalized vertex operator algebra.
we conclude that $[f_{i},{Q}_{j,\lambda}]=0$ for $1\leq i \leq n$. 
Now let us consider the case $s\geq 2$. 
Using (\ref{commutatorformu}), we have 
\begin{align*} 
&[f_{i},{Q}_{j,\lambda}]=0\quad \text{for} \;\; i\neq j,\\
&[f_{i},{Q}_{i,\lambda}]=\tfrac{p}{p-1}\int_{[\Gamma_{s_i+1}]}\tfrac{d}{dz}Q^i(z_{1},...,z_{s_j})dz_{1}...dz_{s_j},    
\end{align*}
where $\tfrac{d}{dz}$ is the total derivation of $Q^i(z_{i},...,z_{s_i})$. 
The result follows. 
\end{proof}

\begin{lemma}\label{lemma_felder}
For $i\in I$ and $\lambda\in\Lambda^{\sigma_i}$, we have the short exact sequence of $B$-modules
\begin{align*}
0\rightarrow \ker Q_{i,\lambda}
\rightarrow
V_\lambda
\rightarrow
\ker Q_{i,\sigma\ast\lambda}(\sigma_i\uparrow\lambda)
\rightarrow 0.
\end{align*}
Furthermore, $\ker Q_{i,\lambda}$ and $\ker Q_{i,\sigma\ast\lambda}$ are the (maximal) $P_i$-submodules of $V_\lambda$ and $V_{\sigma_i\ast\lambda}$, respectively.
If $\lambda\in\Lambda^{\sigma_i}$, then $V_\lambda$ has the $P_i$-module structure.
\end{lemma}
\begin{proof}
Let $M(0)^{i,\perp}$ and $M(0)^i$ be the rank $r-1$ and rank $1$ Heisenberg VOAs generated by $\{\varpi_{j}|j\not=i\}$ and $\alpha_i$, respectively.
Then we have the orthogonal decomposition
\begin{align*}
V_\lambda
\simeq
& M(0)\otimes e^{\sqrt{p}\lambda}\C[\sqrt{p}Q]
\ (\otimes F)
\\
\simeq
&(M(0)^{i,\perp}\oplus M(0)^i)\otimes e^{\sqrt{p}\lambda}\C[\sqrt{p}Q]
\ (\otimes F)
\\
\simeq
&M(0)^{i,\perp}
\otimes
(\bigoplus_{{\alpha}\in Q/\sim}
M(0)^{i}
\otimes 
e^{\sqrt{p}(\alpha+\lambda)}\C[\sqrt{p}\Z\alpha_i]
\ (\otimes F))\\
=:
&M(0)^{i,\perp}
\otimes
(\bigoplus_{{\alpha}\in Q/\sim}
V_{\alpha+\lambda}^{i})
\end{align*}
of $V_\lambda$, where $\alpha\sim\beta$ iff $\beta\in\alpha+\Z\alpha_i$ and the term $(\otimes F)$ is omitted in the case \ref{non-super case}.
Since $f_i$ and $Q_{i,\lambda}$ commutes with the $M(0)^{i,\perp}$-action, we have the orthogonal decompositions
\begin{align*}
&\ker Q_{i,\lambda}
\simeq 
M(0)^{i,\perp}
\otimes
(\bigoplus_{{\alpha}\in Q/\sim}
\ker Q_{i,\lambda}|_{V_{\alpha+\lambda}^{i}}),
\\ 
&\ker Q_{i,\sigma_i\ast\lambda}
\simeq 
M(0)^{i,\perp}
\otimes
(\bigoplus_{{\alpha}\in Q/\sim}
\ker Q_{i,\sigma_i\ast\lambda}|_{V_{\alpha+\sigma_i\ast\lambda}^{i}})
\end{align*}
of $\ker Q_{i,\lambda}$ and $\ker Q_{i,\sigma_i\ast\lambda}$, respectively.
If we can show that 
$\ker Q_{i,\lambda}|_{V_{\alpha+\lambda}^{i}}$
and
$\ker Q_{i,\sigma_i\ast\lambda}|_{V_{\alpha+\sigma_i\ast\lambda}^{i}}$
have the $SL_2^i$-module structures and the short exact sequence
\begin{align*}
0\rightarrow
\ker Q_{i,\lambda}|_{V_{\alpha+\lambda}^{i}}
\rightarrow
V_{\alpha+\lambda}^{i}
\rightarrow
\ker Q_{i,\sigma_i\ast\lambda}|_{V_{\alpha+\sigma_i\ast\lambda}^{i}}(\sigma_i\uparrow\lambda)
\rightarrow 0
\end{align*}
of $B^i$-modules, then by Lemma \ref{Sug-4.2} and the orthogonal decompositions above, the first-half assertion follows.
Therefore, it suffices to consider the cases of rank $1$.
These cases have already been studied, and from the choice of $\Lambda$ and the definition of $l$, the following can be directly shown.
\begin{itemize}
\item 
In the cases where $\alpha_i$ is long, the assertion follows from \cite[Theorem 1.1,1.2]{AM} for $p=r^\vee m$.
\item 
In the cases of $\eqref{non-super case}$ and $\alpha_i$ is short, the assertion follows from \cite[Theorem 1.1,1.2]{AM} for $p=m$.
\item
In the case of $\eqref{super case}$ and $i=r$, the assertion follows from \cite[Theorem 6.1,6.2]{AM1} for $p=2m-1$.
\end{itemize}
Similarly, the last assertion follows from \cite[Theorem 1.3]{AM} and \cite[Theorem 6.3]{AM1}, respectively.
\end{proof}

\subsubsection{Ramond sector of case \eqref{super case}}
Let us recall the notation in Section \ref{sectionramond}.
In the case \eqref{super case}, the simple current element $R$ is given by
$R=\tfrac{\pi i}{\sqrt{2m-1}}\varpi_r$.
Under the spectral flow twist $\mathcal{S}_R$, the conformal grading is deformed as
\begin{align}\label{ABCdefn}
\begin{split}
L_0
&=\operatorname{Res}_z z^2Y(\Delta(\tfrac{ \varpi_r}{\sqrt{2m-1}},z)\omega_{\mathrm{sh}}+e^{\Delta_x}\omega^{(s)},z)\\
&=\operatorname{Res}_z z^2Y(\omega_{\mathrm{sh}}+\omega^{(s)},z)+\tfrac{1}{16}+\operatorname{Res}_z z^2Y((\tfrac{(\varpi_r)_{(1)}}{ \sqrt{2m-1}}+\tfrac{1}{2}(\tfrac{(\varpi_{r})_{(1)}(\varpi_{r})_{(1)}-\sqrt{2m-1}(\varpi_{r})_{(2)}}{2m-1}))\omega_{\mathrm{sh}},z)\\
&=:\operatorname{Res}_z z^2Y(\omega_{\mathrm{sh}}+\omega^{(s)},z)+\tfrac{1}{16}+A(\alpha_r)_{(0)}+B(\alpha_{r-1})_{(0)}+C,
\end{split}
\end{align}
where $\Delta_x$ is defined in \cite{AM2}. Here $A,B,C\in\mathbb{C}$ are determined solely by the spectral-flow twist fixed above (case~(2), $B_r$, $p=2m-1$); in particular they do not depend on the lattice weight~$\mu$ or on~$\lambda\in\Lambda$.
In the Ramond sector $\mathcal{S}_RV_{\lambda}$, we have 
\begin{align}\label{twistedweight}
\Delta_{\mu}
=
\tfrac{1}{2}|\mu|^2-\sqrt{2m-1}(1-\tfrac{1}{2m-1})(\mu,\rho)+A(\alpha_r,\mu)+B(\alpha_{r-1},\mu)+C+\tfrac{1}{16}
\end{align}
Since a spectral flow twist is invertible, under the appropriate changes, the Ramond sector 
$\{\mathcal{S}_RV_\lambda\}_{\lambda\in\Lambda}$ of case \eqref{super case} also defines a shift system.
In conclusion, we obtain the following theorem.

\begin{theorem}\label{main theorem in Section 2}
Let $\g$ be a finite-dimensional simple Lie algebra and $p\in\Z_{\geq 1}$.
\begin{enumerate}
\item\label{main theorem in Section 2-1}
If $p=r^\vee m$ for $m\in\Z_{\geq 1}$, then we have a shift system $(\Lambda,\uparrow,\{V_\lambda\}_{\lambda\in\Lambda})$, where
\begin{itemize}
\item 
The $W$-module $\Lambda$ is \eqref{the representatives in non-super case}, where the $W$-action is induced from \eqref{W-action in non-super case}.
\item
The shift map $\uparrow$ is defined by $\sigma\uparrow\lambda=\sigma\ast\lambda_\bullet-(\sigma\ast\lambda)_\bullet\in P$.
\item 
The weight $B$-module $V_\lambda$ is \eqref{setup in the non-super case}.
The grading is defined by the conformal grading of \eqref{conformal vector for the case of non-super}.
The short screening operator $Q_{i,\lambda}$ is given by \eqref{short screening operator in non-super case}.
\end{itemize}
\item\label{main theorem in Section 2-2}
If $\g=B_r$ and $p=2m-1$ for $m\in\Z_{\geq 1}$, then we have a shift system $(\Lambda,\uparrow,\{V_\lambda\}_{\lambda\in\Lambda})$, where
\begin{itemize}
\item 
The $W$-module $\Lambda$ is \eqref{the representatives in super case}, where the $W$-action is induced from \eqref{W-action in super case}.
\item
The shift map $\uparrow$ is defined by $\sigma\uparrow\lambda=\sigma\ast\lambda_\bullet-(\sigma\ast\lambda)_\bullet\in P$.
\item 
The weight $B$-module $V_\lambda$ is \eqref{setup in the super case}.
The grading is defined by the conformal grading of \eqref{conformal vector for the case of super}.
The short screening operator $Q_{i,\lambda}$ is given by \eqref{short screening operator in super case}.
\end{itemize}
Furthermore, even if $V_\lambda$ is replaced by the Ramond sector $\mathcal{S}_RV_\lambda$, $(\Lambda,\uparrow,\{\mathcal{S}_RV_\lambda\}_{\lambda\in\Lambda})$
is a shift system.
\end{enumerate}
\end{theorem}
Finally, the following is proved in the same manner as \cite[Lemma 2.18, 2.27]{Sug2} so we omit the proof.
\begin{lemma}
\label{lemm: contragredient dual in our cases}
In Theorem \ref{main theorem in Section 2}\eqref{main theorem in Section 2-1} (resp. \eqref{main theorem in Section 2-2}), $V_{\lambda}^\ast(-2\rho)$ is isomorphic to $V_{w_0\ast\lambda'}(-\rho^\vee)$ (resp. $V_{w_0\ast\lambda'}(-\rho)$) as $B$- and $V_{0}$-modules, where $\lambda'$ is the representative of $-w_0(\lambda)$.
In particular, for $0\leq n\leq l(w_0)$, we have
\begin{align*}
&H^n(G\times_BV_{\lambda}^\ast(-2\rho))\simeq H^n(G\times_BV_{w_0\ast\lambda'}(-\rho^\vee))\\
&\text{(resp. $H^n(G\times_BV_{\lambda}^\ast(-2\rho))\simeq H^n(G\times_BV_{w_0\ast\lambda'}(-\rho)))$}
\end{align*}
as $G$- and $H^0(G\times_BV_0)$-modules.
\end{lemma}

\section{Main results I}
\label{section: main results}
In this section, by applying Theorem \ref{main theorem of part 1} to the shift systems in Theorem \ref{main theorem in Section 2}, we show that the main results of \cite{Sug} and part of \cite{Sug2} also hold for our cases.
After stating the definitions and properties common to the two cases of Theorem \ref{main theorem in Section 2}, we will examine each case in Section \ref{section: non-super cases} and \ref{section: super cases}, respectively.

\begin{definition}
\label{def: multiplet W-algebra}
The \textit{multiplet $W$-(super)algebra} is defined by the vertex operator (super) subalgebra
\begin{align*}
W_{\sqrt{p}Q}
=
\bigcap_{i\in I}\ker Q_{i,0}|_{V_{0}}
\end{align*}
of $V_0$.
For each $\lambda\in\Lambda$, 
$\bigcap_{i\in I}\ker Q_{i,\lambda}|_{V_\lambda}$ is a $W_{\sqrt{p}Q}$-submodule of $V_{\lambda}$.
\end{definition}
On the other hand, in the same manner as \cite[Corollary 2.21]{Sug2}, the $0$-th sheaf cohomology
\begin{align*}
H^0(G\times_BV_0)
\end{align*}
of the sheaf associated with the $G$-equivariant vector bundle $G\times_BV_0$ over the flag variety $G/B$ has the VO(S)A structure induced from $V_0$, and the $n$-th sheaf cohomology $H^n(G\times_BV_\lambda(\mu))$ is an $H^0(G\times_BV_0)$-module (for more detail, see \cite[Section 2.1]{Sug2} and the discussion just before \cite[Corollary 2.21]{Sug2}).
Denote
\begin{align}\label{classicaldecomposition}
H^0(G\times_BV_\lambda)
\simeq
\bigoplus_{\alpha\in P_+\cap Q}
L_{\alpha+\lambda^\bullet}
\otimes
\mathcal{W}_{-\alpha+\lambda}
\end{align}
the decomposition of $H^0(G\times_BV_\lambda)$ as $G$-module, where $L_\beta$ is the irreducible $\g$-module with highest weight $\beta\in P_+$ and $\mathcal{W}_{-\alpha+\lambda}$ is the multiplicity of a weight vector of $L_{\alpha+\lambda^\bullet}$.
In our case, 
\begin{align*}
H^0(G\times_BV_0)^G
=V_0^B
=\bigcap_{i\in I}\ker f_i|_{V_0^{h=0}}
=\mathcal{W}_0
\end{align*} 
is a vertex operator (super) subalgebra of $H^0(G\times_BV_0)$ and each 
$\mathcal{W}_{-\alpha+\lambda}$
is a $\mathcal{W}_0$-module.

Let us use the same notation and setup in Theorem \ref{main theorem in Section 2}.
The proofs are basically the same as in \cite{Sug,Sug2}, so we will only outline them.
By Theorem \ref{main theorem of part 1} and Theorem \ref{main theorem in Section 2} we obtain the following.
\begin{theorem}\label{main theorem in Section 3}
Let $(\Lambda,\uparrow,\{V_\lambda\}_{\lambda\in\Lambda})$ be the shift system in Theorem \ref{main theorem in Section 2}\eqref{main theorem in Section 2-1}(resp. \eqref{main theorem in Section 2-2}).
\begin{enumerate}
\item\label{main theorem 3-1}
The evaluation map 
\begin{align*}
\operatorname{ev}\colon H^0(G\times_BV_\lambda)
\rightarrow
\bigcap_{i\in I}\ker Q_{i,\lambda},\ \ s\mapsto s(\operatorname{id}_{G/B})
\end{align*}
is injective, and is isomorphic iff $\lambda\in\Lambda$ satisfies the condition \eqref{weak condition}.
In particular, for any $p\in \Z_{\geq 2}$, we have the isomorphism of VOAs
$H^0(G\times_BV_{\sqrt{p}Q})\simeq W_{\sqrt{p}Q}$.
\item\label{main theorem 3-1(2)}
For $\lambda\in\Lambda$ such that $(p\lambda_\bullet+\rho^\vee,{}^L\theta)\leq p$ (resp. $(p\lambda_\bullet+\rho,{}^L\theta)\leq p$),
we have a natural $G$- and $W_{\sqrt{p}Q}$-module isomorphism
\begin{align*}
H^n(G\times_BV_\lambda)\simeq H^{n+l(w_0)}(G\times_BV_{w_0\ast\lambda}(-\rho^\vee))\\
(\text{resp.}\quad H^n(G\times_BV_\lambda)\simeq H^{n+l(w_0)}(G\times_BV_{w_0\ast\lambda}(-\rho))).
\end{align*}
In particular, $H^{n>0}(G\times_BV_\lambda)\simeq 0$ and $H^0(G\times_BV_\lambda)\simeq H^0(G\times_BV_{-w_0(\lambda)})^\ast$ as $W_{\sqrt{p}Q}$-modules.
\end{enumerate}
\end{theorem}
\begin{proof}
It suffices to show the last assertion.
By Lemma \ref{lemm: strong condition for our cases}, Lemma \ref{lemm: contragredient dual in our cases} and Serre duality, we have
\begin{align*}
H^0(G\times_BV_\lambda)
&\simeq 
H^{l(w_0)}(G\times_BV_{w_0\ast\lambda}(-\rho'))\\
&\simeq
H^{l(w_0)}(G\times_BV_{-w_0(\lambda)}^\ast(-2\rho))\\
& \simeq
H^0(G\times_BV_{-w_0(\lambda)})^\ast
\end{align*}
and thus the claim is proved, where $\rho'$ in the second term is $\rho^\vee$ (case \eqref{non-super case}) or $\rho$ (case \eqref{super case}), respectively.
\end{proof}

\subsection{The non-super case \eqref{non-super case}}
\label{section: non-super cases}

By applying Corollary \ref{Sug2-3-15} and Lemma \ref{lemma: general simplicity theorem for lambda is zero}, 
we obtain the following:
\begin{corollary}
\label{character formula for non-super case}
\cite[Theorem 1.2, Lemma 4.21]{Sug}
For $\alpha\in P_+\cap Q$ and $\lambda\in\Lambda$ such that $(p\lambda_\bullet+\rho^\vee,{}^L\theta)\leq p$, 
\begin{align}
\label{character of non-super case}
\ch_q\mathcal{W}_{-\alpha+\lambda}
=
\sum_{\sigma\in W}(-1)^{l(\sigma)}
\ch_q V_{\sigma\ast\lambda}^{h=\alpha+\lambda^\bullet-\sigma\uparrow\lambda}
=\sum_{\sigma\in W}(-1)^{l(\sigma)}
\tfrac{q^{\frac{1}{2p}|-p\sigma(\alpha+\lambda^\bullet+\rho)+p\lambda_\bullet+\rho^\vee|^2}}{\eta(q)^r}.
\end{align}
\end{corollary}

\begin{corollary}
\label{vacuum simplicity for non-super case}
\cite[Section 3.2]{Sug2}
For $m\geq{}^Lh^\vee-1$ and $\lambda=\lambda^\bullet$, $W_{\sqrt{p}Q}$ and $W_{\sqrt{p}P}$ are simple as VOA and generalized VOA, respectively, and $W_{\sqrt{p}(Q-\lambda^\bullet)}$ is simple as $W_{\sqrt{p}Q}$-module.
Furthermore, for $m\geq{}^Lh^\vee-1$ and $\beta\in P_+$, $\mathcal{W}_{-\beta}$ is simple as $\mathcal{W}_0$-module.
\end{corollary}

\begin{remark}\label{rem:AM-C2}
When $\mathfrak{g}=B_2$ and $p$ is even, Corollary~\ref{character formula for non-super case} yields the
conjectural $q$-characters in \cite[Conjecture~8.1]{AM3}, and Corollary~\ref{vacuum simplicity for non-super case}
establishes the corresponding simplicity conjectured in \cite[Conjecture
~2.1]{AM3}.
\end{remark}

\subsection{The super case \eqref{super case}}
\label{section: super cases}
In the same manner as above, we obtain the following:
\begin{corollary}
\label{character formula for super case}
For $\alpha\in P\cap Q$ and $\lambda\in \Lambda$ such that $(p\lambda_\bullet+\rho,{}^L\theta)\leq p$, we have
\begin{align}
&\ch_q\mathcal{W_{-\alpha+\lambda}}=\ch_q F \sum_{\sigma\in W}(-1)^{l(\sigma)} \ch_q V_{\sigma\ast \lambda}^{h=\alpha+\lambda^\bullet-\sigma\uparrow \lambda}=\ch_q F\sum_{\sigma\in W}(-1)^{l(\sigma)}\tfrac{q^{\Delta_{\sqrt{p}(-(\alpha+\lambda^\bullet)+\sigma\ast \lambda_\bullet)}}}{\eta(q)^r},\label{character1}\\
&\mathrm{sch}_q\mathcal{W}_{-\alpha+\lambda}=\mathrm{sch}_q F \sum_{\sigma\in W}(-1)^{l(\sigma)} \mathrm{sch}_q V_{\sigma\ast \lambda}^{h=\alpha+\lambda^\bullet-\sigma\uparrow \lambda}=\mathrm{sch}_q F\sum_{\sigma\in W}(-1)^{l(\sigma)+f(\sigma\circ(\alpha+\lambda^\bullet ))} \tfrac{q^{\Delta_{\sqrt{p}(-(\alpha+\lambda^\bullet)+\sigma\ast \lambda_\bullet)}}}{\eta(q)^r},\label{character2}\\
&\ch_q\mathcal{S}_R\mathcal{W}_{-\alpha+\lambda}=q^{C+\tfrac{1}{16}}\ch_q\iota^*F \sum_{\sigma\in W}(-1)^{l(\sigma)}\tfrac{q^{\Delta_{\sqrt{p}(-(\alpha+\lambda^\bullet)+\sigma\ast \lambda_\bullet)}+(A\alpha_r+B\alpha_{r-1},-\sqrt{p}(\sigma\circ (\alpha+\lambda^\bullet)-\lambda_\bullet)}}{\eta(q)^r}, \label{character3}
\end{align}
where for any $\beta\in P$, define a function $f:P\mapsto \mathbb{Z}$ by $f(\beta)=\lfloor(\beta,\alpha_r) \rfloor$, $A, B,C$ are defined in \eqref{ABCdefn}, and $\operatorname{sch}_qF$ and $\ch_q F$ are in 
Section \ref{Section: free field algebra}.
We denote 
$H^0(G\times_B\mathcal{S}_RV_\lambda)=:\bigoplus_{\alpha\in P_+\cap Q}L_{\alpha+\lambda^\bullet}\otimes\mathcal{S}_R\mathcal{W}_{-\alpha+\lambda}$
by abuse of notation.
\end{corollary}

\begin{corollary}\label{vacuum simplicity for super case}
When $p\geq h^\vee-1$ and $\lambda=0$, $W_{\sqrt{p}Q}$ is simple as a VOA and a generalized VOA, respectively. 
Furthermore, for $\alpha\in P_+\cap Q$, $\mathcal{W}_{-\alpha}$ is simple as $\mathcal{W}_0$-module.
\end{corollary}

\begin{remark}
For simply-laced $\g$ (resp. rank $1$ case), modularity of the characters were already 
studied in \cite{BrM} (resp. \cite{AM1,AM2}). 
We expect similar results to hold in our cases too. \end{remark}

\section{Main results II}
\label{section: simplicity theorem under the strong condition}
In Section \ref{section: simplicity theorem under the strong condition}, first we show that $\mathcal{W}_0$ is isomorphic to the corresponding principal W-algebra $\mathbf{W}^k(\g)$.
It enables us to use another type of character formula, i.e., Kazhdan--Lusztig type character formula (see \cite[Theorem 1.1]{KT}, \cite[Theorem 7.7.1]{Ar}) and the discussion in Section \ref{section: simplicity revisit} (see the last two paragraphs).
Using these results, we will attempt to extend the simplicity theorems (Corollary \ref{vacuum simplicity for non-super case} and Corollary \ref{vacuum simplicity for super case}) to $\lambda\in\Lambda$ satisfying condition \eqref{strong condition} (or, equivalently, the conditions in Lemma \ref{lemm: strong condition for our cases}). 
In the non-super case \eqref{non-super case}, the simplicity theorem will be proved in the same manner as \cite{Sug2}.
On the other hand, the super case \eqref{super case} has not been studied as much as the case \eqref{non-super case}, and we will prove the simplicity theorem under a particular assumption (exactness of $+$-reduction).
Since the flow of the discussion is the same as that of \cite[Section 3.3]{Sug2}, details are sometimes omitted.

\subsection{Preliminary from W-algebra}
\label{subsection: preliminary from W-algebra}
Let us recall the notation in \cite[Chapter 6]{Kac2} (see also \cite[Chapter 1]{Wak}) and \cite{Ar,ArF,ACL}.
For an affine Lie algebra $\hat{\g}$, the Cartan subalgebra $\hat{\h}$ and its dual $\hat{\h}^\ast$ are decomposed as
$\hat{\h}=\h\oplus(\C K+\C d)$ and $\hat{\h}^\ast=\h^\ast\oplus(\C\delta+\C\Lambda_0)$, respectively.
For $\hat\mu\in\hat{\h}^\ast$, we have
\begin{align*}
\hat\mu=\mu+\langle \hat\mu,K\rangle\Lambda_0+(\mu,\Lambda_0)\delta,
\ \ 
(\mu\in\h^\ast).
\end{align*}
By \cite[Proposition 6.5]{Kac2}, the affine Weyl group $\hat{W}$ is $W\ltimes Q^\vee$ (for $\hat\g=X^{(1)}_r$) and $W\ltimes Q_{B_r}$ (for $\hat\g=A^{(2)}_{2r}$), respectively.
The $\hat{W}$-action on $\hat{\h}^\ast$ is given by
\begin{align*}
\sigma t_\beta(\hat\mu)
=
\sigma(\mu+\langle\hat{\mu},K\rangle\beta)
+
\langle\hat{\mu},K\rangle\Lambda_0
+
((\hat\mu,\Lambda_0-\beta)-\tfrac{1}{2}|\beta|^2\langle\hat\mu,K\rangle)\delta.
\end{align*}
For $\hat\rho:=\rho+h^\vee\Lambda_0$, the circle $\hat{W}$-action on $\hat{\h}^\ast$ is given by $w\circ\hat\mu:=w(\hat\mu+\hat\rho)-\hat\rho$.

For $k\in\C$, denote $\hat{\h}^\ast_k$ the set of $\hat\mu\in\hat{\h}^\ast$ such that $(\hat{\mu},K)=k$.
Let us define 
$V^k(\g)=U(\hat\g)\otimes_{\g[t]\oplus\C K}\C_k$
the universal affine vertex algebra at level $k$, where $\g[t]$ (resp. $K$) act on $\C_k$ trivially (resp. $k\operatorname{id}$).
More generally, we can define the Weyl module $\hat{V}_{\beta,k}=U(\hat\g)\otimes_{\g[t]\oplus\C K}L_{\beta}$ over $V^k(\g)$ in the same manner.
Clearly, we have $\hat{V}_{0,k}=V^k(\g)$.
For $\hat\mu\in\hat{\h}^\ast_k$, define $\hat{M}(\hat\mu)=U(\hat{\g})\otimes_{U(\hat{\mathfrak{b}})}\C_{\hat\mu}$ the Verma module and its simple quotient $\hat{L}(\hat\mu)$, respectively, where $U(\hat{\mathfrak{n}})$ (resp. $\h$, $K$) act on $\C_{\hat\mu}$ trivially (resp. by the character of $\mu$, $k\operatorname{id}$).

By applying the Drinfeld--Sokolov reduction $H^0_{\mathrm{DS}}(-)$ to $V^k(\g)$, the (universal) principal W-algebra $\mathbf{W}^k(\g)$ at level $k$ and its unique simple quotient $\mathbf{W}_k(\g)$ are obtained.
Note that we have the Feigin--Frenkel duality
\begin{align}
&\mathbf{W}^k(\g)\simeq \mathbf{W}^{{}^Lk}({}^L\g),\ \ 
r^\vee(k+h^\vee)({}^Lk+{}^Lh^\vee)=1,\ \ \g=X_r^{(1)},\label{FF duality for non-super case}\\
&\mathbf{W}^k(\osp(1|2r))\simeq \mathbf{W}^{{}^Lk}(\osp(1|2r)),\ \ 
4(k+h^\vee)({}^Lk+h^\vee)=1\label{FF duality for super case}
\end{align}
in our cases (see \cite{FF,FF1}, \cite[Remark 6.5]{Gen1}).
More generally, for $\check{\mu}\in\check{P}$ in the coweight lattice $\check{P}=Q^\ast$, we have the twisted Drinfeld--Sokolov reduction $H^0_{\mathrm{DS},\check{\mu}}(-)$ and the \textit{Arakawa--Frenkel module} 
\begin{align*}
\mathbf{T}^{k+h^\vee}_{\beta,\check{\mu}}
:=H^0_{\mathrm{DS},\check{\mu}}(\hat{V}_{\beta,k})
\end{align*}
over $\mathbf{W}^{k}(\g)$.
For the Langlands dual ${}^L\g$, we can define
$\check{\mathbf{T}}^{{}^Lk+{}^Lh^\vee}_{\check{\mu},\beta}=H^0_{\mathrm{DS},\beta}(\hat{V}_{\beta,{}^Lk})$
in the same manner. 
Then for $k\not\in Q$ or the case in \cite{Sug2}, the Feigin--Frenkel duality $\mathbf{T}^{k+h^\vee}_{\beta,\check{\mu}}\simeq \check{\mathbf{T}}^{{}^Lk+{}^Lh^\vee}_{\check{\mu},\beta}$ under the identification $\mathbf{W}^k(\g)\simeq \mathbf{W}^{{}^Lk}({}^L\g)$ is known to hold (see \cite{ArF,Sug2}).
Let $\chi_{\mu}:Z(\mathfrak{g})\rightarrow \mathbb{C}$ with $\mu\in\mathfrak{h}^\ast$ be the map in \cite[(27)]{ACL}, where $Z(\mathfrak{g})$ is the center of the universal enveloping algebra of $U(\mathfrak{g})$. 
Then one can define the Verma module $\mathbf{M}_k(\chi_\mu):=U(\mathbf{W}^k(\g))\otimes_{U(\mathbf{W}^k(\g)_{\geq 0})}\C_{\chi_\mu}$ over $\mathbf{W}^k(\g)$ with highest weight $\chi_\mu$ and its simple quotient $\mathbf{L}_k(\chi_\mu)$ in the same manner above (in the Langlands dual case, denote $\check{\mathbf{M}}_k(\chi_\mu)$ and $\check{\mathbf{L}}_k(\chi_\mu)$, respectively).

\subsection{The non-super case \eqref{non-super case}}
\label{subsection: simplicity theorem for non-super case}
Let us consider our case \eqref{non-super case}.
By the injectivity of the Miura map \cite[Lemma 5.4]{Gen1} and the consideration in \cite[Remark 4.1]{FKRW}, we have
\begin{align*}
\mathbf{W}^k(\g)\hookrightarrow\bigcap_{i\in I}\ker e^{-\frac{1}{{k+h^\vee}}b_i}_{(0)}|_{M(0)},
\end{align*}
where $b_i$ satisfies $(b_i,b_j)=(k+h^\vee)(\alpha_i,\alpha_j)$.
In particular, Fock module $M(\mu)$ is also a $\mathbf{W}^k(\g)$-module.
By comparing the conformal vector \eqref{conformal vector for the case of non-super} and \cite[(3.22)]{ArF}, for $k=\frac{1}{p}-h^\vee$ and $b_i=-\frac{1}{\sqrt{p}}\alpha_i$, we have
\begin{align}
\label{embedding of W-algebra into W_0}
\mathbf{W}^{m-{}^Lh^\vee}({}^L\g)
\overset{\text{\eqref{FF duality for non-super case}}}{\simeq}
\mathbf{W}^{\frac{1}{p}-h^\vee}(\g)
\hookrightarrow
\bigcap_{i\in I}\ker e^{\sqrt{p}\alpha_i}_{(0)}|_{M(0)}=\mathcal{W}_0.
\end{align}
On the other hands, by \cite[Section 4.3]{ArF}, we have
\begin{align}
\label{character of Arakawa-Frenkel module}
\ch_q T^{m}_{p\lambda_\bullet,\alpha+\lambda^\bullet}
=
\ch_q
T^{\frac{1}{p}}_{\alpha+\lambda^\bullet,p\lambda_\bullet}
=
\ch_q\mathcal{W}_{-\alpha+\lambda}
\end{align}
for $\alpha\in P_+\cap Q$ and $\lambda\in\Lambda$ satisfying the strong condition Lemma \ref{lemm: strong condition for our cases}\eqref{lemm: strong condition for non-super case}.
In particular, by applying $\alpha=\lambda=0$, we have $\ch_q\mathbf{W}^{m-{}^Lh^\vee}({}^L\g)=\ch_q\mathbf{W}^{\frac{1}{p}-h^\vee}(\g)=\ch_q\mathcal{W}_0$,
and thus the injection in \eqref{embedding of W-algebra into W_0} is isomorphic.
By Corollary \ref{vacuum simplicity for non-super case}, $\mathcal{W}_{-\beta}$ is simple as $\mathbf{W}^{m-{}^Lh^\vee}({}^L\g)$-module for any $\beta\in P_+$.
In particular, for $\beta=0$,
we have
\begin{align}
\label{W_0 and W-algebra coincides in non-super case}
\mathcal{W}_0
\simeq 
\mathbf{W}^{m-{}^Lh^\vee}({}^L\g)
\simeq
\mathbf{W}_{m-{}^Lh^\vee}({}^L\g)
\simeq
\mathbf{W}^{\frac{1}{p}-h^\vee}(\g)
\simeq
\mathbf{W}_{\frac{1}{p}-h^\vee}(\g).
\end{align}
In the remainder of this subsection, we use
\begin{align}
\label{k and check-k in non-super case}
k=\tfrac{1}{p}-h^\vee,\ \ \check{k}=m-{}^Lh^\vee.
\end{align}
On the other hand, in the same manner as \cite[Lemma 3.14, 3.19]{Sug2}, we have the following.

\begin{lemma}\label{lemma: top component}
Let $\alpha\in P_+\cap Q$, $\lambda\in\Lambda$, and let $X$ be any one of the four modules
\[
M(\sqrt p(-\alpha+\lambda)),\qquad W_{-\alpha+\lambda},\qquad
T^{\frac1p}_{\alpha+\lambda_\bullet,\,p\lambda_\bullet},\qquad
\check T^{\,m}_{p\lambda_\bullet,\,\alpha+\lambda_\bullet}.
\]
Then the top component of $X$ is one-dimensional, and it is isomorphic to
$\mathbb C_{\chi_{\alpha-\lambda}}$ as a module over
$\mathrm{Zhu}(\mathcal W^{k}(\mathfrak g))$, and to
$\mathbb C_{\chi_{p(-\alpha+\lambda)}}$ as a module over
$\mathrm{Zhu}(\mathcal W^{\check k}({}^L\mathfrak g))$, with $k,\check k$ as in (\ref{k and check-k in non-super case}).
In particular, if $\operatorname{ch}_q X=\operatorname{ch}_q L_k(\chi_{\alpha-\lambda})$
$\bigl(=\operatorname{ch}_q \check L_{\check k}(\chi_{p(-\alpha+\lambda)})\bigr)$, then
$X\cong L_k(\chi_{\alpha-\lambda})$ as $\mathcal W^{k}(\mathfrak g)$-modules, and
$X\cong \check L_{\check k}(\chi_{p(-\alpha+\lambda)})$ as
$\mathcal W^{\check k}({}^L\mathfrak g)$-modules.
\end{lemma}

We aim to apply the discussion in Section \ref{section: simplicity revisit} (see the last two paragraphs).
First, let us give a KL-type decomposition in Lemma \ref{new simplicity theorem}\eqref{formal KL decomposition}.
Unless otherwise noted, all symbols (e.g., $Q$, $\alpha_i$,...) are the ones of $\g$.
Denote $h_\mu$ and $\check{h}_\mu$ the top conformal weight of $\mathbf{M}_{k}(\chi_\mu)$ and $\check{\mathbf{M}}_{\check{k}}(\chi_\mu)$, respectively (see \cite[(30)]{ACL}).
By \eqref{conformal weight of lattice point vector}, for any $\beta\in(\sqrt{p}Q)^\ast$ (in other words, $\sqrt{p}\beta\in P_{{}^L\g}$), we have
\begin{align}
\label{compare the Delta and h}
\Delta_{\beta}
=h_{-\frac{1}{\sqrt{p}}\beta}
=\check{h}_{\sqrt{p}\beta}.
\end{align}
Let us take $\beta=\sqrt{p}(-\alpha+\lambda)$ ($\alpha\in P_+\cap Q$, $\lambda\in\Lambda$).
Then the corresponding highest weight of $V^{\check{k}}({}^L\g)$ is
\begin{align*}
-p(\alpha+\lambda^\bullet+\rho)+p\lambda_\bullet+\check{k}\Lambda_{0,{}^L\g}.
\end{align*}
Let us recall the notation in \cite{KT}.
In our case (i.e., $\widehat{{}^L\g}$), we have
\begin{align*}
&\hat{W}_{{}^L\g}
=W_{{}^L\g}\ltimes(Q_{{}^L\g})^\vee
=W\ltimes r^\vee Q,\\
&{\Delta}^{\mathrm{re}}_{+,\widehat{{}^L\g}}
=
\{\alpha_{\pm}+n\delta\ |\ 
\alpha_{\pm}\in\Delta_{+,{}^L\g},\ n\in\Z_{>0}\}
\sqcup
\Delta_{+,{}^L\g}
=
\{\alpha_{\pm}^\vee+n\delta\ |\ 
\alpha_{\pm}\in\Delta_+,\ n\in\Z_{>0}\}
\sqcup
\Delta_{+}^\vee
\\
& 
({\Delta}^{\mathrm{re}}_{+,\widehat{{}^L\g}})^\vee
=\{r^\vee\alpha_{\pm}^l+r^\vee n\delta\ |\ \alpha_{\pm}^l\in\Delta_{\pm}^l,\ n\in\Z_{>0}\}
\sqcup
\{r^\vee\alpha_{\pm}^s+n\delta\ |\ \alpha_{\pm}^s\in\Delta_{\pm}^s,\ n\in\Z_{>0}\}
\sqcup
r^\vee\Delta_{+},
\\
&
\begin{aligned}
{\mathcal{C}}^{+}_{\widehat{{}^L\g}}&=
\{\nu\ |\ (\gamma^\vee,\nu+\hat{\rho}_{{}^L\g})_{{}^LQ}\geq 0\ \text{for any $\gamma\in\Delta^{\mathrm{re}}_{+,\widehat{{}^L\g}}$}
\}
\\
&=
\{\nu\ |\ (\gamma^\vee,\nu+\rho^\vee+{}^Lh^\vee\Lambda_{0,{}^L\g})_{{}^LQ}\geq 0\ \text{for any $\gamma^\vee\in(\Delta^{\mathrm{re}}_{+,\widehat{{}^L\g}})^\vee$}
\}.
\end{aligned}
\end{align*}
By \cite[Lemma 2.10]{KT}, there exists unique $\sigma t_{r^\vee\beta}\in\hat{W}_{{}^L\g}=W\ltimes r^\vee Q$ such that
\begin{align*}
\sigma t_{r^\vee\beta}
\circ_{\widehat{{}^L\g}}
(-p(\alpha+\lambda^\bullet+\rho)+p\lambda_\bullet+\check{k}\Lambda_{0,{}^L\g})\in{\mathcal{C}}^{+}_{\widehat{{}^L\g}}.
\end{align*}
In other words, we have
\begin{align*}
&0\leq 
(\sigma^{-1}(\Delta_+^l),p(\beta-(\alpha+\lambda^\bullet+\rho))+p\lambda_\bullet+\rho^\vee)_Q
\leq r^\vee m,\\
\land\ 
&0\leq 
(\sigma^{-1}(\Delta_+^s),p(\beta-(\alpha+\lambda^\bullet+\rho))+p\lambda_\bullet+\rho^\vee)_Q
\leq m.
\end{align*}
We show that if $\lambda$ satisfies the condition Lemma \ref{lemm: strong condition for our cases}\eqref{lemm: strong condition for non-super case}, i.e. $(p\lambda_\bullet+\rho^\vee,\theta_s)\leq m$, such $\sigma t_{r^\vee\beta}$ is independent of the choice of $\lambda_\bullet$.
Let us consider the case $\g=B_r$ (other cases are similar).
Set
$\beta-(\alpha+\lambda^\bullet+\rho)=\sum_{i\in I}n_i\varpi_i$ ($n_i\in\Z$).
For $\alpha_{+}\in\Delta_+$, if there exists, denote $a_{\pm\alpha_{+}}\in\Delta_{+}$ such that $\sigma^{-1}(a_{\pm\alpha_{+}})=\pm\alpha_{+}$.
For any $\alpha_+\in\Delta_+$, one of $a_{\alpha_+}$ or $a_{-\alpha_+}$ always exists.
Set
$I_{\pm}=\{i\ |\ \exists a_{\pm\alpha_i}\}$.
Clearly, we have $I=I_+\sqcup I_{-}$.
For $i\in I_+$, we have
\begin{align*}
0\leq (\alpha_i,p\sum_{i\in I}n_i\varpi_i+p\lambda_\bullet+\rho^\vee)=pn_i(\alpha_i,\varpi_i)+\lambda_i\leq 
\begin{cases}
r^\vee m&(\alpha_i\in\Delta_+^l),\\
m&(\alpha_i\in\Delta_+^s),
\end{cases}
\end{align*}
and thus $n_i=0$ for $i\in I_+$.
Similarly, we have $n_i=-1$ for $i\in I_{-}$, and thus $\sum_{i\in I}n_i\varpi_i=-\sum_{i\in I_-}\varpi_i$.
Next, let us consider $a_{\pm\theta_s}$.
If $a_{\theta_s}$ exists, then we have $I_{-}=\phi$.
On the other hand, if $a_{-\theta_s}$ exists, then we have $I_{-}=\{r\}$.
Each of these cases corresponds to the case $\alpha+\lambda^\bullet+\rho\in Q$ or not, which is independent of the choice of $\lambda_\bullet$.
Therefore, we can use the notation $y_{\alpha,\lambda^\bullet}\in \hat{W}$ for the inverse of the unique element $\sigma t_\beta$ in the above discussion.
By \cite[Theorem 1.1]{KT}, we obtain the following:
\begin{lemma}\label{uniqueness of KL}
For $\alpha\in P_+\cap Q$ and $\lambda\in\Lambda$ in \eqref{the representatives in non-super case} s.t. $(p\lambda_\bullet+\rho^\vee,{}^L\theta)\leq p$, we have
\begin{align*}
\ch_q\hat{L}_{\check{k}}(-p(\alpha+\lambda^\bullet+\rho)+p\lambda_\bullet+\check{k}\Lambda_{0,{}^L\g})_{{}^L\g}
=
\sum_{y\geq y_{\alpha,\lambda^\bullet}}
a_{y,y_{\alpha,\lambda^\bullet}}
\ch_q\hat{M}_{\check{k}}(y\circ_{\widehat{{}^L\g}}\mu_\lambda)_{{}^L\g},
\end{align*}
where
$a_{y,w}=Q_{y,w}(1)$ for the inverse Kazhdan-Lusztig polynomial $Q_{y,w}(q)$ and 
\begin{align}\label{the element in C+ in on-super case}
\mu_{\lambda}
=
y_{\alpha,\lambda^\bullet}^{-1}
\circ_{\widehat{{}^L\g}}
(-p(\alpha+\lambda^\bullet+\rho)+p\lambda_\bullet+\check{k}\Lambda_{0,{}^L\g}).
\end{align}
Furthermore, by applying the (exact) $+$-reduction functor (see \cite[Theorem 7.7.1]{Ar}), we have
\begin{align*}
\ch_q\check{\mathbf{L}}_{\check{k}}(\chi_{p(-\alpha+\lambda)})
=
\sum_{y\geq y_{\alpha,\lambda^\bullet}}
a_{y,y_{\alpha,\lambda^\bullet}}
\ch_q\check{\mathbf{M}}_{\check{k}}(\chi_{y\circ_{\widehat{{}^L\g}}\mu_\lambda+p\rho-\check{k}\Lambda_{0,{}^L\g}}).
\end{align*}
In particular, under the setting \eqref{setup for lem 2.6} below, the condition Lemma \ref{new simplicity theorem}\eqref{formal KL decomposition} holds.   
\end{lemma}

Let us apply the discussion in Section \ref{section: simplicity revisit} to our case.
In Lemma \ref{new simplicity theorem}, set 
\begin{align}
\label{setup for lem 2.6}
\begin{split}
&\beta=\alpha+\lambda^\bullet,
\ \ 
(\lambda_0,\lambda_1)=(0,\lambda_\bullet),
\ \ 
a_{y,\beta}=Q_{y,y_{\alpha,\lambda^\bullet}}(1),\\
&\text{$\mu_{\lambda_i}$ is \eqref{the element in C+ in on-super case} for each case}, 
\ \ 
y_\sigma=t_{r^\vee(\alpha+\lambda^\bullet-\sigma\circ(\alpha+\lambda^\bullet))}y_{\alpha,\lambda^\bullet}\ \ (\sigma\in W),\\
&{\mathbb{M}}(y,\mu_{\lambda_i})
=
\check{\mathbf{M}}_{\check{k}}(\chi_{y\circ_{\widehat{{}^L\g}}\mu_\lambda+p\rho-\check{k}\Lambda_{0,{}^L\g}}),
\ \ 
{\mathbb{L}}(-\beta+\lambda_i)
=
\check{\mathbf{L}}_{\check{k}}(\chi_{p(-\alpha+\lambda)}),
\end{split}
\end{align}
respectively.
Then the conditions Lemma \ref{new simplicity theorem}\eqref{formal KL decomposition},\eqref{simplicity at lambda0} are already satisfied.
On the other hand, the condition Lemma \ref{new simplicity theorem}\eqref{Verma coincidence condition},\eqref{Verma and staggered B-module} follows from \eqref{compare the Delta and h} and \cite[(29)]{ACL}.
Then by Lemma \ref{new simplicity theorem} and Lemma \ref{lemma: top component}, $\mathcal{W}_{-\alpha+\lambda}$ is simple as $\mathcal{W}_0$-module.
In particular, each $\mathcal{W}_{-\alpha+\lambda}$ is a cyclic $\mathcal{W}_0$-module.
Finally, by applying Lemma \ref{lemma: general simplicity theorem for lambda is nonzero}, the simplicity of $H^0(G\times_BV_\lambda)$ also holds.
In conclusion, we have the following:

\begin{theorem}
\label{simplicity theorem in non-super case}
Let us consider the non-super case \eqref{non-super case} and  
$H^0(G\times_BV_\lambda)\simeq \bigoplus_{\alpha\in P_+\cap Q}L_{\alpha+\lambda^\bullet}\otimes \mathcal{W}_{-\alpha+\lambda}$.
Then, for \eqref{k and check-k in non-super case} and $m\geq {}^Lh^\vee-1$, we have the VOA-isomorphism
\begin{align*}
\mathcal{W}_0
\simeq
\mathbf{W}^{\check{k}}({}^L\g)
\simeq
\mathbf{W}_{\check{k}}({}^L\g)
\simeq
\mathbf{W}^{k}(\g)
\simeq
\mathbf{W}_{k}(\g)
\end{align*}
and each $\mathcal{W}_{-\alpha+\lambda^\bullet}$ is simple as $\mathcal{W}_0$-module.
More generally, for $\lambda\in\Lambda$ such that $(p\lambda_\bullet+\rho^\vee,{}^L\theta)\leq p$, 
\begin{align*}
\mathcal{W}_{-\alpha+\lambda}
&\simeq
\check{\mathbf{L}}_{\check{k}}(\chi_{p(-\alpha+\lambda)})
\simeq
\check{\mathbf{T}}^{m}_{p\lambda_\bullet,\alpha+\lambda^\bullet}
\ \ 
\text{(as $\mathbf{W}^{\check{k}}({}^L\g)$-modules)}\\
&\simeq
\mathbf{L}_{k}(\chi_{\alpha-\lambda})
\simeq 
\mathbf{T}^{\frac{1}{p}}_{\alpha+\lambda^\bullet,p\lambda_\bullet}
\ \ 
\text{(as $\mathbf{W}^{k}(\g)$-modules)},
\end{align*}
and $H^0(G\times_BV_\lambda)$ is simple as $H^0(G\times_BV_0)$-module, respectively.
\end{theorem}

\begin{remark}\label{Section4-Remark4.4}
It is expected that the simple modules of $H^0(G\times_BV_0)$ are classified by $\Lambda$. 
Constructing simple modules for general $\lambda\in\Lambda$ 
seems to require the development of not only our geometric approach, but also the W-algebraic analysis of the Verma module $\mathbf{M}_k(\chi_\mu)=H^0_{\mathrm{DS},\check{\mu}}(\hat{M}(\hat{\lambda}))$.
The latter is currently being pursued by T. Arakawa and the second author to investigate the tensor category structure of W-algebras at generic levels.
\end{remark}

\subsection{The super case \eqref{super case}}
\label{subsection: simplicity theorem for super case}
In this subsection, set 
\begin{align} 
\label{k and check-k in super case}
{k}=\tfrac{1}{2(2m-1)}-(r+\tfrac{1}{2}),\ \ 
\check{k}=m-r-1,\ \ 
h^\vee=h^\vee_{\osp(1|2r)}=r+\tfrac{1}{2}.
\end{align}
Then $4(k+h^\vee)(\check{k}+h^\vee)=1$ and
by \cite[Theorem 6.3]{Gen1} and Corollary \ref{vacuum simplicity for super case}, we have
\begin{align}
\label{W_0 and W-algebra coincides in super case}
\mathcal{W}_0
\simeq 
\mathbf{W}^{k}(\osp(1|2r))
\simeq
\mathbf{W}_{k}(\osp(1|2r))
\simeq
\mathbf{W}^{\check{k}}(\osp(1|2r))
\simeq
\mathbf{W}_{\check{k}}(\osp(1|2r))
\end{align}
and each $\mathcal{W}_{-\alpha}$ ($\alpha\in P_+\cap Q$) is simple as $\mathcal{W}_0$-module.
In the same manner as above, we have the following:

\begin{lemma}\label{topspace1}
The top component of $M(-\sqrt{p}(\alpha+\lambda))$ and $\mathcal{W}_{-\alpha+\lambda}$ are 
$\C_{\chi_{\alpha-\lambda}}$ as $\operatorname{Zhu}(\mathbf{W}_{{k}}(\mathfrak{osp}(1|2r))$-modules
and
$\mathbb{C}_{\chi_{p(-\alpha+\lambda)}}$
as $\operatorname{Zhu}(\mathbf{W}_{\check{k}}(\mathfrak{osp}(1|2r))$-modules, respectively.
In particular, if the character of these modules coincides with that of $\mathbf{L}_{{k}}(\chi_{-\alpha+\lambda})$ or $\mathbf{L}_{\check{k}}(\chi_{p(-\alpha+\lambda)})$, then they are isomorphic to $\mathbf{L}_{{k}}(\chi_{-\alpha+\lambda})$ or $\mathbf{L}_{\check{k}}(\chi_{p(-\alpha+\lambda)})$.
\end{lemma}

We aim to apply the discussion in Section \ref{section: simplicity revisit} (see the last two paragraphs).
First let us consider Lemma \ref{new simplicity theorem} in our case ($\lambda_0=0$, $\lambda_1=\lambda$ such that $\lambda^\bullet=0$ and $(p\lambda_\bullet+\rho,{}^L\theta)\leq p$).
The condition Lemma \ref{new simplicity theorem}\eqref{simplicity at lambda0} is already checked in the above discussion.
By \cite{KRW} and the results in Section \ref{Section: free field algebra}, we have 
\begin{align*}
\ch_q\mathbf{M}_{\check{k}}(\chi_{\mu})
=
\tfrac{\ch_qF}{\eta(q)^r}{q^{\frac{1}{2p}|\mu-p\rho+\frac{1}{p}\rho|^2}},\ \ 
\text{in particular, 
$\ch_q\mathbf{M}_{\check{k}}(\chi_{p(-\alpha+\lambda)})
=
\ch_qV_{\lambda}^{h=\sigma\circ(\alpha+\lambda^\bullet)}$
}
\end{align*}
and the condition Lemma \ref{new simplicity theorem}\eqref{Verma coincidence condition},\eqref{Verma and staggered B-module} are satisfied.

Lastly, we want to check the Lemma \ref{new simplicity theorem}\eqref{formal KL decomposition}.
However, to the authors' knowledge, the KL type decomposition as \cite{KT,Ar} has not been proven in our case $\osp(1|2r)$.
In this paper, we will prove the KL type decomposition (and Lemma \ref{new simplicity theorem}\eqref{formal KL decomposition}) in the case of \textit{twisted affine Lie algebra of type $A^{(2)}_{2r}$}, 
and prove the simplicity theorem for super-case \eqref{super case} under the assumption that the characters in these two types are the same under the coordinate change.
Let us recall the data in $\hat{\g}'=A_{2r}^{(2)}$:\footnote{Note that here we use the second description in \cite{Car}. On the other hand, \cite{Kac2, Wak} uses the first description of \cite{Car}.
The classical part is different for these two descriptions: For the first one (resp. second one), it is type $C_n$ (resp. type $B_n$).}
\begin{align*}
&\hat{W}_{\hat{\g}'}=W\ltimes Q,
\ \ 
h^\vee_{A_{2r}^{(2)}}=2r+1=2h^\vee_{\osp(1|2r)},
\\
&{\Delta}^{\mathrm{re}}_{+,\hat{\g}'}
=\{
\alpha_{\pm}+n\delta
\ |\ 
\alpha_{\pm}\in\Delta_{\pm},\ n\in\Z_{>0}
\}
\sqcup
\{2\alpha_{\pm}^s+(2n+1)\delta
\ |\ 
\alpha_{\pm}^s\in\Delta_{\pm}^s,\ n\in\Z_{\geq 0}\}
\sqcup
\Delta_{+},\\
&{\mathcal{C}}^+_{\hat{\g}'}
=
\{
\mu\in\hat\h^\ast
\ |\ 
(\gamma^\vee,\mu+\hat{\rho}_{A_{2r}^{(2)}})\in \Z_{\geq 0}\ \text{for any $\gamma\in{\Delta}^{\mathrm{re}}_{+,\hat{\g}'}$}
\},
 \ \ 
\hat{\rho}_{A_{2r}^{(2)}}=\rho+h^\vee_{A_{2r}^{(2)}}\Lambda_0^c,
\end{align*}
where all classical terms are of type $B_n$ (see \cite[Section 6]{Kac2}), and $\Lambda_0^c=\frac12\Lambda_0$.
By \cite[Lemma 2.10]{KT}, there exists unique $\sigma t_{\beta}\in\hat{W}_{\hat{\g}'}=W\ltimes Q$ such that
\begin{align*}
&\sigma t_\beta\circ_{\hat{\g}'}(-p(\alpha+\lambda^\bullet+\rho^\vee)+p\lambda_\bullet+\check{k}\Lambda_0)\\
&=
\sigma(p(\beta-\alpha-\lambda^\bullet-\rho^\vee)+p\lambda_\bullet+\rho)+p\Lambda_0^c-\hat{\rho}_{A_{2r}^{(2)}}
\in
{\mathcal{C}}^+_{\hat{\g}'}.
\end{align*}
In other words, we have
\begin{align*}
0\leq (p(\beta-\alpha-\lambda^\bullet-\rho^\vee)+p\lambda_\bullet+\rho,\sigma^{-1}(\alpha_+^\vee))\leq p\ \ \text{for any $\alpha_+\in\Delta_+$}.
\end{align*}
In the same manner as Section \ref{subsection: simplicity theorem for non-super case}, we can  determine such a ($\lambda_\bullet$-independent) $\sigma t_\beta$ explicitly as follows:
\begin{align*}
\begin{cases}
(\sigma,\beta)=(\operatorname{id},\alpha+\rho^\vee),
\ \ 
\mu_\lambda=p\lambda+\check{k}\Lambda_0
&(\lambda^\bullet=0)\\
(\sigma,\beta)=(\sigma_r,\alpha+\rho^\vee),
\ \ 
\mu_\lambda=\sigma_r\circ_{\hat{\g}'}(p\lambda+\check{k}\Lambda_0)
&(\lambda^\bullet=\varpi_r)
\end{cases}
\end{align*}
By combining it with \cite[Theorem 1.1]{KT}, we obtain the KL type decomposition in $A_{2r}^{(2)}$ case.
\begin{lemma}\label{KL character formula for twisted A_2n}
For $\alpha\in P_+\cap Q$ and $\lambda\in\Lambda$ in \eqref{the representatives in super case} s.t. $(p\lambda_\bullet+\rho,{}^L\theta)\leq p$, we have
\begin{align*}
\ch_q\hat{L}(-p(\alpha+\lambda^\bullet+\rho^\vee)+p\lambda_\bullet+\check{k}\Lambda_0)_{\hat{\g}'}
=
\sum_{y\geq y_{\alpha,\lambda^\bullet}}
a_{y,y_{\alpha,\lambda^\bullet}}
\ch_q\hat{M}(y\circ_{\hat{\g}'}\mu_\lambda)_{\hat{\g}'},
\end{align*}
where
$a_{y,w}=Q_{y,w}(1)$ for the inverse Kazhdan-Lusztig polynomial $Q_{y,w}(q)$ and
\begin{align*}
y_{\alpha,\lambda^\bullet}
=
\begin{cases}
t_{-\alpha-\rho^\vee}&(\lambda^\bullet=0),\\
t_{-\alpha-\rho^\vee}\sigma_r& (\lambda^\bullet=\varpi_r),
\end{cases}
\ \ 
\mu_\lambda
=
\begin{cases}
p\lambda+\check{k}\Lambda_0& (\lambda^\bullet=0),\\  
\sigma_r\circ_{\hat{\g}'}(p\lambda+\check{k}\Lambda_0)& (\lambda^\bullet=\varpi_r).
\end{cases}
\end{align*}
\end{lemma}

Let us relate type $A_{2r}^{(2)}$ and $\osp(1|2r)$.
Because the second description of $A_{2n}^{(2)}$ in \cite{Car} has the same Dynkin diagram as the one of type $B(0,r)=\osp(1|2r)$, replacing a black node with a white node, these two types are closely related. 
In particular, according to \cite[Section 9.5]{KW3}, the character formulas for the modules of affine Lie (super)algebras of type $A_{2r}^{(2)}$ and of type $B(0,r)$ are the same up to coordinate change.
In fact, choosing the coordinate of $h\in \hat{\mathfrak{h}}$ of twisted affine Kac-Moody Lie algebra of type $A_{2r}^{(2)}$:
\begin{align*}
    h:=2\pi i(-\tau\Lambda_0^c+z+u\delta).
\end{align*}
Then by proper coordinate transformation $t$, one can obtain that
\begin{align}\label{transformation formula}
    \ch_q[\hat{M}(A_{2r}^{(2)})](h)=\ch_q[\hat{M}(B(0,r)](t(h)).
\end{align}
Meanwhile, for type $A_{2n}^{(2)}$, the KL type decomposition was already proved in Lemma \ref{KL character formula for twisted A_2n} above. 
Therefore, we also have the same KL type decomposition on the $\osp(1|2r)$-side.

To apply Lemma \ref{new simplicity theorem}\eqref{formal KL decomposition} to our super case \eqref{super case}, one must move from affine to the principal $W$-algebra.
We denote the category $\mathcal{O}$ for $L_\kappa(\mathfrak{osp}(1|2r))$ by $\mathcal{O}_\kappa$ ($\kappa\in\C$).
Let $\mathcal{O}_{\kappa}^{[\lambda]}$ be the full subcategory of $\mathcal{O}_\kappa$ whose objects have their local composition factors isomorphic to $\hat{L}_{\kappa}(w\circ \lambda)$.
The exactness of the $+$-reduction $H^0_{\mathrm{DS},+}(\cdot)$ is already proved in \cite{Ar} (and thus the KL type character formula of $\mathbf{W}^\kappa(\g)$-module is derived from \cite[Theorem 1.1]{KT}), but in our case it has not known yet to the authors' knowledge.
\begin{conjecture}\label{exact+reduction'}
For $\check{k}=m-r-1$ and $\lambda\in\Lambda$ $(p\lambda_\bullet+\rho,{}^L\theta)\leq p$, the $+$-reduction functor $H^0_{\mathrm{DS},+}(\cdot)$ defines an exact functor from $\mathcal{O}_{\check{k}}^{[\lambda]}$ to $\mathbf{W}_{\check{k}}(\mathfrak{osp}(1|2r))\text{-$\mathrm{mod}$}$.
\end{conjecture}
If we assume the above conjecture, then we have the KL-type decomposition in our case $\mathbf{W}^{\check{k}}(\osp(1|2r))$ and the last condition Lemma \ref{new simplicity theorem}\eqref{formal KL decomposition} is also satisfied.
The remaining discussion is the same as Section \ref{subsection: simplicity theorem for non-super case}, and the conclusion is as follows:

\begin{theorem}
\label{simplicity theorem in super case}
Let us consider the super case \eqref{super case}, and set 
$H^0(G\times_BV_\lambda)\simeq \bigoplus_{\alpha\in P_+\cap Q}L_{\alpha+\lambda^\bullet}\otimes \mathcal{W}_{-\alpha+\lambda}$.
Then, for \eqref{k and check-k in super case} and
$m\geq r$, we have the VOSA-isomorphism
\begin{align*}
\mathcal{W}_0
\simeq
\mathbf{W}^{\check{k}}(\osp(1|2r))
\simeq
\mathbf{W}_{\check{k}}(\osp(1|2r))
\simeq
\mathbf{W}^{k}(\osp(1|2r))
\simeq
\mathbf{W}_{k}(\osp(1|2r))
\end{align*}
for $\check{k}=m-r-1$, $k=\tfrac{1}{2(2m-1)}-(r+\tfrac{1}{2})$,
and each $\mathcal{W}_{-\alpha}$ is simple as  $\mathcal{W}_0$-module ($\alpha\in P_+\cap Q$).
Furthermore, under the assumption of Conjecture \ref{exact+reduction'}, for any $\alpha\in P_+\cap Q$ and $\lambda\in\Lambda$ such that $\lambda^\bullet=0$ and $(p\lambda_\bullet+\rho^\vee,{}^L\theta)\leq p$, 
\begin{align*}
\mathcal{W}_{-\alpha+\lambda}
&\simeq
{\mathbf{L}}_{\check{k}}(\chi_{p(-\alpha+\lambda)})
\ \ 
\text{(as $\mathbf{W}^{\check{k}}(\osp(1|2r))$-modules)}\\
&\simeq
\mathbf{L}_{k}(\chi_{\alpha-\lambda})
\ \ 
\text{(as $\mathbf{W}^{k}(\osp(1|2r))$-modules)},
\end{align*}
and $H^0(G\times_BV_\lambda)$ is simple as $H^0(G\times_BV_0)$-module, respectively.
Finally, if one can check the simplicity for one $\lambda\in\Lambda$ such that $\lambda^\bullet=\varpi_r$ and $(p\lambda_\bullet+\rho,{}^L\theta)\leq p$, then the simplicity theorem for $\mathcal{W}_{-\alpha+\lambda}$ and $H^0(G\times_BV_\lambda)$ above holds for any $\lambda\in\Lambda$ satisfying the same condition.
\end{theorem}

\begin{remark}\label{rem:4.8-r1}
When $r=1$, the multiplet algebra $\mathcal{W}_0$ is the $N=1$ triplet vertex
operator superalgebra studied in \cite{AM2}. In this case,
Theorem~\ref{simplicity theorem in super case} was already established in
\cite{AM2}.
\end{remark}

\end{document}